\documentclass[11pt]{article}

\usepackage{amsmath,amsthm,verbatim,amssymb,amsfonts,amscd, graphicx}
\usepackage{graphicx}
\usepackage{mathrsfs}
\usepackage{enumerate}
\usepackage{mathtools}
\usepackage{lipsum}         
\usepackage{xargs}                      
\usepackage{enumitem}
\usepackage{authblk}
\usepackage{cite}
\usepackage{hyperref}
\usepackage{bbm}
\usepackage{float}
\usepackage[pdftex,dvipsnames]{xcolor}  
\usepackage{caption}
\usepackage{subcaption}
\usepackage{siunitx}
\usepackage[colorinlistoftodos,prependcaption,textsize=tiny]{todonotes}

\theoremstyle{plain}

\newtheorem{remark}{Remark}
\newtheorem{proposition}{Proposition}

\theoremstyle{definition}
\newtheorem{definition}{Definition}

\usepackage{bm}

\newcommand{\R}{\mathbb{R}}

\begin{document}

\title{Imaging of seabed topography from the scattering of water waves}

\author{Adrian Kirkeby}
\affil{Simula Research Laboratory, Norway \\
\small{\href{adrian@simula.no}{adrian@simula.no} }}
\maketitle

\begin{abstract}
We consider the problem of reconstructing the seabed topography from observations of surface gravity waves. We formulate the problem as a classical inverse scattering problem using the mild-slope equation and analyze the topographic dependence of the forward and inverse problem. Moreover, we propose a useful model simplification that makes the inverse problem much more tractable. As water waves allow for observations of the full wave field, it differs quite a lot from the classical inverse scattering problems, and we utilize this to prove a conditional stability result for the inverse problem. Last, we develop a simple and fast numerical inversion method and test it on synthetic data to verify our analysis. 
\end{abstract}

\section{Introduction and mathematical model}

The task of mapping the ocean depth has a long and interesting history (cf. \cite{dierssen2014bathymetry}), and accurate maps of the seabed topography (also called bathymetry) are important for safe sea travel, off-shore-based industry and infrastructure, geological exploration, oceanography, and many other purposes  \cite{wolfl2019seafloor}. The most common method for depth estimation is sonar, where the bottom reflects an acoustic pulse, and the depth is computed from the travel time. This method is robust and accurate but requires scanning large ocean parts with a vessel and can be both time-consuming and dangerous, especially in coastal and shallow waters \cite{lyzenga1985shallow,wolfl2019seafloor}. 

Another method to infer the seabed topography is by using satellite/airborne observations. Such techniques are classified as either passive or active (cf. \cite{shah2020review} for a recent review). In passive methods, one relies on pure (possibly multi-spectral) optical data, where one can directly observe the bottom through the water column. Based on knowledge of the optical properties of water and various other physical processes, the local depth can be calculated from the pixel intensity \cite{lyzenga2006multispectral}. In active methods like LiDAR or SAR, focused electromagnetic radiation penetrates the water, and the recorded scattered field is used to reconstruct the topography \cite{lyzenga1985shallow,wiehle2019automatic}. These methods can be very effective, but a drawback is that their success depends on the microscopic properties of the water; if the water has a high concentration of scattering and absorbing particles like algae, bubbles, minerals, etc., the incoming electromagnetic waves are quickly absorbed as they penetrate the water, and the reflected signal contains very little information about the water column and the bottom \cite{shah2020review}. 

The approach to topography estimation taken in this paper is quite different; we do not try to bypass the barrier posed by the water to get direct information about the bottom; instead, we analyze the causal relationship between the bottom topography and the surface waves, and aim to reconstruct the topography from observations of the waves. This approach is not new, and there seem to be two main approaches to solving it. The first, which we will call the dispersion method, utilizes the linear dispersion relation for linear waves over constant depth $d$, i.e., the formula $\omega^2 = gk\tanh(kd)$. Assuming that this relation remains valid also for waves over variable depth $d(X)$, one then uses various methods to extract local estimates of the spatially varying wavenumber $k$ or the wave velocity $c = \omega/k$, from video or images the surface waves. Given such data, one proceeds by inverting the dispersion relation to get an estimate of $d(X)$. An early reference for this approach, going back to World War II, is \cite{williams1947determination}. More recent work is found in, e.g., \cite{piotrowski2002accuracy,stockdon2000estimation,plant2008ocean,holman2013cbathy}.  

The second strategy is based on partial differential equations (PDEs), and we will refer to it as the PDE method. From an inverse problems point of view, this is a very natural strategy: one assumes some known PDE models the physics and that the unknown quantity is some parameter (a coefficient, source term, boundary, or initial condition) in that PDE. The measurement is (part of) the solution to the PDE, and the inverse problem is posed as inverting the PDE to find the unknown quantity. For example, in \cite{vasan2013inverse}, the authors consider the inversion of the fully non-linear water wave equations: this is a formidable problem, and they show that it is possible to determine the topography from measuring the wave amplitude, velocity, and acceleration, and conduct 1D reconstructions numerically. They also note that the problem is severely ill-posed but do not analyze this further. In \cite{fontelos2017bottom}, the non-linear problem is studied as an optimization problem, and it is shown that one can uniquely determine the topography from observations of the wave amplitude and velocity potential. Also here, ill-posedness is discussed, but no conclusions are reached. Other papers using PDE methods, some of which will be mentioned later, are \cite{nicholls2009detection,vasan2021ocean,wu2023adjoint,lecaros2020stability,gessese2011reconstruction}.

The benefit of the dispersion methods is that they are simple to use, quite robust, and require only image/video data. The drawback is that they are based on the rather \emph{ad hoc} local dispersion relation, so it is unclear under which physical conditions this model is suitable. Moreover, due to the uncertainty principle for localized frequency estimation (cf. \cite{grochenig2013foundations}), there are limitations to how accurately one can estimate the local wavenumber and the spatial resolution of the reconstructed data will be limited by this \cite{piotrowski2002accuracy}. 

The benefit of PDE methods is that PDEs can capture much more of the complex behavior of water waves, i.e., serve as better models for the physics, and therefore also allow for more accurate determination of the topography. The drawback is that the PDE models are much more complicated to analyze and solve numerically and require very accurate measurements of physical quantities that are not easily measured, like surface acceleration or surface velocity potential. 

\subsection*{Contribution}
The approach taken in this paper lies somewhere in between the dispersion methods and the PDE methods. We model the surface wave using a PDE that depends on the topography in a non-linear, dispersive\footnote{By dispersive dependence, we mean that waves of different wavelengths depend differently on the topography.} manner, but eventually, we will end up inverting a dispersion relation to estimate the topography. This approach has several advantages; first of all, it is clear from the beginning what assumptions are made for the model to be valid, and so the applicability and limitations of the model are clear. Second, we can give a rather precise stability estimate on the reconstructed topography in the case of noisy measurements. As far as we know, no such estimates have been given previously for the topography inversion problem. Moreover, we assume that we can measure the wave amplitude only, and our analysis results in a simple and fast reconstruction method based on a rigorous simplification of the underlying PDE that does not rely on localized Fourier- or wavelet analysis. The drawback is that the limitations of the model are somewhat restrictive in that the depth should not vary very fast with respect to the wavelength and that it is a linear model, so the waves should not be too steep. \\

The paper is organized as follows: In Section \ref{sect: mathematical model}, we introduce the mathematical model and the underlying assumptions. In Section \ref{sect: wave scattering}, we first analyze how our model depends on the topography. We then formulate and analyze the (forward) scattering problem for water waves by seabed topography. Moreover, we show that for longer waves, one can simplify the PDE while only introducing a small error; this is valuable as it greatly simplifies the inverse problem. In Section \ref{sect: inversion}, we analyze the inverse problem; Propositions \ref{prop: uniqueness} and \ref{prop: stability} are the main results of the paper and show uniqueness and conditional Hölder-stability of the inverse problem. In Section \ref{sect: numerical}, we propose a simple, fast inversion method and test it on synthetic data, and in Section \ref{sect: conclusion}, we summarize our findings and propose some ideas for further inquiry.

\subsection{Mathematical model}
\label{sect: mathematical model}
We consider an incompressible, irrotational fluid occupying a region 
$$ \Omega = \{ (X,z) : X = (x_1,x_2) \in \R^2, \quad - d(X) < z < 0 \}.$$
Above, the function $d(X): \R^2 \to \R$ describes the pointwise water depth, i.e., the seabed topography, and it is the mathematical object that we are interested in estimating. 
We assume that $d(X)$ is on the form $d(X) = H_0 - h(X)$, where $H_0$ is the mean depth and $h(X)$ accounts for the topography. Moreover, we assume $h(X)$ is zero outside some compact set $\Omega$ and that $d(X)  > 0 $. We will refer to $d(X)$ as the topography or depth. 

Under the above conditions, the fluid velocity can be described by the gradient of a harmonic potential $\Phi(X,z,t)$, and the linearized equations governing the fluid-wave system read (cf. \cite{waterwavesprob,ablowitz2011nonlinear})
\begin{equation}
    \begin{cases}
    \partial_t\eta - \partial_z \Phi|_{z = 0} &= 0 , \quad X \in \R^2,\\
     \partial_t \Phi|_{z= 0} + g\eta &= 0, \quad X \in \R^2,\\
    \qquad \qquad   \Delta\Phi &= 0, \quad (X,z) \in \R^2\times(-d(X),0), \quad \partial_\nu \Phi = 0, \quad (X,z) \in \R^2 \times -d(X). 
    \end{cases}
    \label{eq:lin surf}
\end{equation}
Here $\eta(X,t)$ is the surface wave amplitude, $\partial_\nu$ is the upward pointing normal derivative and $g$ is the gravitational acceleration. For a wave $\eta \sim a \exp(\bm{k}\cdot X - ct)$ with wavenumber $\bm{k} \in \R^2$ and amplitude $a$, the linearization of the free-surface Euler equations above is obtained under the assumption that $ak \approx |\nabla \eta| \ll 1$, where $k = |\bm{k}|$, and $a \ll \text{min } d(X)$. That is, the waves should not be to steep, and the amplitude should be smaller than the minimal depth.  Figure \ref{fig: 3dwave} illustrates the situation. 

\begin{figure}
    \centering
    \includegraphics[width = 0.9\textwidth]{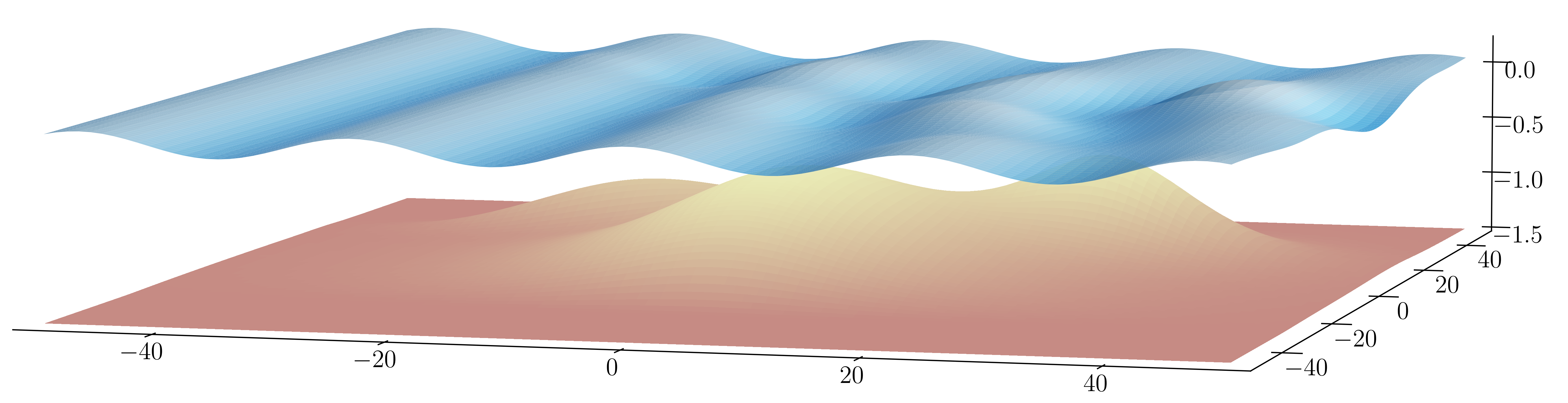}
    \caption{Plot of an incoming wave scattered by the topography. The wave is the solution to the mild-slope equation with the shown topography. The units on all axes are meters. Note that the aspect ratio is skewed for illustration purposes.}
    \label{fig: 3dwave}
\end{figure}

\subsection{Waves over variable topography and the mild-slope equation}
It is common in the analysis of water waves to reduce equation \eqref{eq:lin surf} to a non-local system of equations on the boundary by introducing the Dirichlet-to-Neumann (DN) operator $G(d)$, defined as 

\begin{align*}
G(d)\varphi = \partial_z \Phi|_{z= 0} \quad \text{where } \Phi \text{ solves } 
\begin{cases}
    &\Phi|_{z = 0} = \varphi, \quad  \Delta\Phi = 0 \quad (X,z) \in \Omega, \\
    &\partial_\nu \Phi = 0 \quad X \in \R^2, z = -d(X).
\end{cases}  
\end{align*}
We can then express the water waves equations in the so-called Zakharov-Craig-Sulem formulation \cite{waterwavesprob}:
\begin{align*} 
\begin{cases}
    \partial_t \eta &= G(d)\varphi, \\
    \partial_t \varphi &= -g\eta,
\end{cases}
\end{align*}
equipped with suitable initial conditions. Many people have analyzed the influence of topography on water waves using this formulation. For example, the paper \cite{craig2005hamiltonian} studies the effect of rough, periodic bottoms on (non-linear) long waves, and in \cite{andrade2018three,milewski1998formulation}, an explicit construction of the operator is presented. Moreover, in the paper \cite{nicholls2009detection}, the topography inversion problem in 1D is studied by using an expansion method for the DN operator. We refer to Chapter 3 in \cite{waterwavesprob}) for more details on the properties of depth dependence of the DN operator. 

However, the non-local nature of the DN operator and its strongly non-linear dependence on the topography makes for a less-than-optimal choice of model for the inverse problem. Instead, we choose to use the so-called \emph{mild-slope equation} (MSE). The MSE is an elliptic equation describing the scattering of a time-harmonic wave over slowly varying topography. More explicitly, it is derived from equation \eqref{eq:lin surf} under the assumption that 
\begin{equation}
     \frac{|\nabla d|}{d} \leq \delta k, \quad \text{where} \quad  \delta \ll 1. 
    \label{eq: mild slope condition}
\end{equation}
This condition says that the spatial variation in depth should be small relative to the wavelength: in terms of the wavelength $\lambda = 2\pi/k$ it reads $\frac{|\nabla d|}{2\pi d} \lambda \ll 1$, i.e., the ratio $$\frac{\text{bottom steepness}\times \text{wavelength}}{\text{depth}}$$ should be small for the MSE to be a good approximation.   

To derive the MSE, one first assumes that the wave and velocity potential is time-harmonic with frequency $\omega$, i.e., $\eta(X,t) = \tilde{\eta}(X)\exp(i\omega t)$ and $ \Phi(X,z,t) = \tilde{\Phi}(X,z)\exp(i\omega t).$ By taking $\Phi$ to be separable of the form $\varphi(X)f(X,z)$ with $f(X,z) = \cosh(k(z+d(X))/\cosh(kd(X))$, where $k$ satisfies $\omega^2 = gk\tanh(kd(X))$, one can then depth-integrate \eqref{eq:lin surf} and discard $\mathcal{O}(\delta^2)$ terms to find that the spatial component of $\tilde{\eta}$ satisfies
\begin{equation}
    \nabla \cdot c \nabla \tilde{\eta} + k^2c \tilde{\eta} = 0.
    \label{eq: mse}
\end{equation}
This is the mild-slope equation. The coefficient $k^2$ is the square of the magnitude of the local wave number $k$, and $k$ satisfies the dispersion relation 
\begin{equation}
    \omega^2 = gk(X) \tanh(k(X)d(X)). 
    \label{eq: dispersion}
\end{equation} 
The coefficient $c = c(X)$ is given by 
\begin{equation}
    c = \frac{\omega^2}{2k^2(X)c_0}\left(1 + \frac{2k(X)d(X)}{\sinh(2k(X)d(X))} \right), 
    \label{eq: c}
\end{equation}
where $c_0 $ is a constant\footnote{This is just a re-scaling of the MSE such that it is the constant coefficient Helmholtz equation outside the area of varying depth.} such that $c = 1$ outside $\Omega$. 
For a detailed derivation, analysis, and discussion of \eqref{eq: mse}, see \cite{dingemans2000water} or \cite{mei1989applied}.
In \cite{booij1983note}, the mild-slope equation was shown to have excellent agreement with finite element solutions of \eqref{eq:lin surf} for $\delta \leq 1/3$, and for $\delta \leq 1$ when the depth contours were parallel to the direction of wave propagation. We note that there exist different modifications to the MSE for steeper slopes (cf. \cite{porter2003mild,mei1989applied}) and that these models depend on the topography in increasingly intricate ways.

Notice that the quantity of interest in our inverse problem, i.e., the topography $d(X)$, only enters indirectly through the coefficients $k^2$ and $c$. 

\section{Water wave scattering}
\label{sect: wave scattering}
We can divide the scattering of water waves into two main phenomena: One is diffraction, which is an effect of the boundaries of the surface on the wave. Examples are the ``bending" of a wave past a corner or the spherical expansion of a wave as it passes through a small hole. The second is refraction; refraction is the changes in wave properties due to changes in the medium in which the wave propagates. An example of refraction that can easily be observed is that ocean waves often appear to approach the shore at a right angle: this is not because all ocean waves propagate with wavefronts parallel to the shore initially, but because the decrease in local wave velocity due the decrease in depth causes the wavefront to become parallel to the beach \cite{holthuijsen2010waves}. 

In this work, we are interested in refraction, especially refraction due to changes in the water depth. 
\newpage
\subsection{Topographic dependence in the MSE}
\label{section: depth dependence}

\emph{``The effect of the wave virtually disappears at a depth of about one-half wavelength, and a submarine at such depth is hardly aware of the heavy waves on the surface."} \\ 
-- Walter Munk, Waves across the Pacific, Documentary.\\
\vspace{2mm}

To get a sense of how water waves are influenced by topography, it is useful to consider the constant depth setting first. 
For a monochromatic wave with angular frequency $\omega$ propagating in the direction $\bm{k} = (k_1, k_2)$ on water with depth $H_0$, we have 
$$ \eta(X,t) = a \exp( i\bm{k}\cdot X - \omega t),$$
where the magnitude of the wavenumber $k = |\bm{k}|$ satisfies $\omega^2 = gk\tanh(kH_0)$. The corresponding bulk water velocity $U =\nabla \Phi $ is 
$$ U(X,z) = (u,v,w)^{T} =  \frac{ga}{\omega} \exp(i \bm{k}\cdot X - \omega t)  \begin{bmatrix} k_1\frac{\cosh(k(z+H_0))}{\cosh(kH_0)} \\ k_2\frac{\cosh(k(z+H_0))}{\cosh(kH_0)} \\- ik\frac{\sinh(k(z+H_0))}{\cosh(kH_0)} \end{bmatrix}.
$$
Considering, for example, the ratio of $|u(z)|/|u(0)|$ of the horizontal velocity, we find that it decays rapidly with depth towards the value $\cosh(kH_0)^{-1} \approx \exp(-kH_0)$. At $z = \lambda/2$, we have that 
$$|u(z)|/|u(0)| \approx \frac{\exp(\pi-kH_0)}{2} \approx 0 \quad \text{for} \quad kH_0 = \frac{2\pi H_0}{\lambda} \gg \pi. $$
Hence, the subsurface fluid velocity decays rapidly with depth, and the motion associated with shorter waves (larger $k$) decays faster. The kinematic interaction with the bottom at such depths is therefore negligible, and the bottom will have very little influence on the fluid motion in the water and, therefore, on the surface waves.  

In the MSE, the topography influences the waves through the coefficients $k^2$ and $c$. We recall 
\begin{equation*}
    k^2 :  \omega^2 = gk\tanh(kd) \quad \text{and}\quad c  = \frac{\omega^2}{2k^2c_0}\left(1 + \frac{2kd}{\sinh(2kd)} \right). 
\end{equation*}
Neither $c$ nor $k$ are independent parameters in the problem but functions of the topography $d$ and the frequency $\omega$. Figure \ref{fig: kofd} shows a plot of $k$ as a function of depth for some values of $\omega$, and we see that the wavenumber decays rapidly for lower depths, eventually tending toward a constant value as $d$ increases. Moreover, the decay is more rapid for higher frequencies. 

\begin{figure}
    \centering
    \includegraphics[width = 0.8\textwidth]{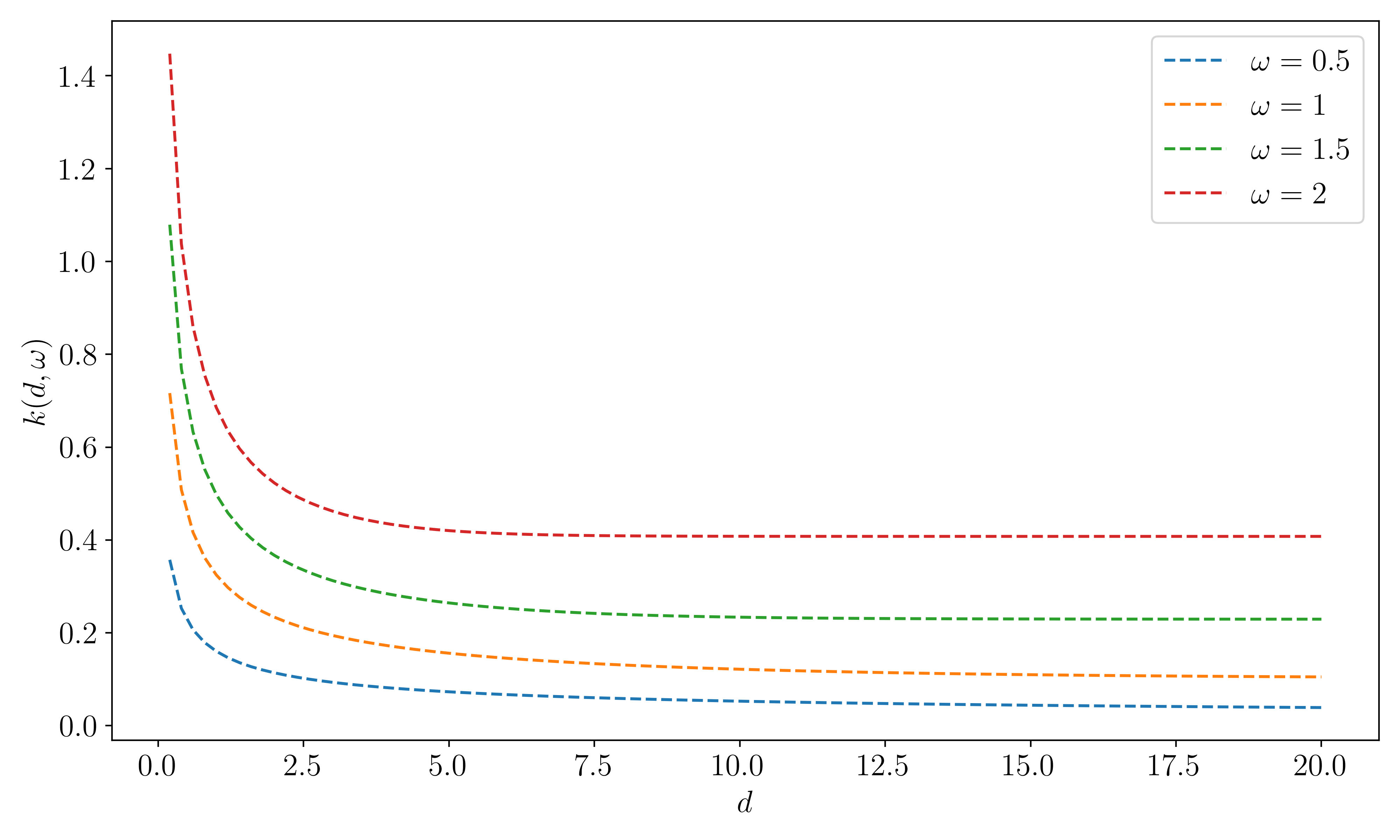}
    \caption{The figure shows the wavenumber $k$ as a function of depth (in meters), for $\omega$ in the range $0.5$ to $2$. We see that $k$ approaches the value $\mu = \omega^2/g$ as $d$ increases. }
    \label{fig: kofd}
\end{figure}

Denoting $\mu = \omega^2/g$, we have 
$$d(k) = \frac{\tanh^{-1}(\mu/k)}{k}. $$
We note the following properties of the wavenumber-to-depth function. For a given $\mu$, and for $\alpha > 0$, let ${I_\alpha = [\mu + \alpha, \infty)}$.
\begin{itemize}
    \item[i)] $d(k)$ is a bijection on the interval $I_\alpha$ and so $d$ is uniquely determined $k$ by $k \in I_a$. The case $k = \mu$ corresponds to $\tanh(kd) = 1  \implies d = \infty$. 
    \item[ii)] $d(k)$ is Lipschitz continuous on $I_\alpha$ and the Lipschitz constant $M_{\alpha,\mu}$ satisfies
    \begin{equation}
         \frac{\mu}{2(\mu + \alpha)^2\alpha}  \leq M_{\alpha,\mu} \leq \frac{3}{(\mu + \alpha)\alpha}. 
         \label{eq: lipschitz}
    \end{equation}
    Hence the Lipschitz constant is rapidly increasing as $\alpha \to 0$, showing that the reconstruction of the depth $d(X)$ from $k(X)$ becomes increasingly unstable as $k \to \mu$. Moreover, for fixed $\alpha$, $M_{\alpha,\mu}$ is a decreasing function of $\mu$. 
\end{itemize}
The two statements above are deduced by noting that $$d'(k) = k^{-2}\left(\frac{\mu k}{\mu^2 - k^2} - \tanh^{-1}(\mu/k)\right)$$
is negative and strictly increasing on $I_a$, and applying standard inequalities.  

\begin{remark}
    Writing \eqref{eq: dispersion} as 
    $$ \mu = k \tanh(kd), $$
    we see that $\mu \to k$ occurs when $\tanh(kd)\to 1$. This is the manifestation of the previously discussed diminishing influence of the bottom on the waves when either the depth is too great, the wave number is too large, or both. For a numerical example, assume $kd = 3$ and $\mu = 1$. Then $\tanh(3) \approx 0.995$ and $\mu - k  \approx - 0.005$. Assuming we have measured some noisy $k_\varepsilon = k + \varepsilon$ for some small $\varepsilon > 0$, the error in the reconstruction $d$ from $k_\varepsilon$ satisfies 
    $$ |d(k) - d(k_\varepsilon)| \leq \varepsilon M_{0.005,1} \leq  \varepsilon \times 101.$$
    This is quite a hopeless bound, as the noise is amplified 100 times. 
\end{remark}

Even though the coefficient $c$ is redundant with respect to the unique reconstruction of $d$, it might carry information that stabilizes the inversion $ (k,c) \mapsto d$. 
However, this is not the case. Combining \eqref{eq: dispersion} and \eqref{eq: c}, we find
\begin{equation*}
    d(k,c) = \left(\frac{ck}{\omega^2} - \frac{1}{2k}\right)\sinh\left(2\tanh^{-1}\left(\mu/k\right)\right) = \left(\frac{ck}{\omega^2} - \frac{1}{2k}\right)\left( \frac{2\mu k}{k^2 - \mu^2}\right) ,
\end{equation*}
and the local Lipschitz constant satisfies 
\begin{equation*}
      \|\nabla_{k,c} d \|_\infty \geq \frac{k}{\omega^2} \left( \frac{2\mu k}{k^2 - \mu^2}\right),  
\end{equation*}
which is not better than the reconstruction from $k$ alone when $k \to \mu$.

\subsection{The scattering problem}

We now formulate the problem of the scattering of water waves by topography. We consider the following situation: an incoming wave ${ \eta_{i} \sim \exp(i \bm{k}_i \cdot X)}$ satisfying ${(\Delta + k_i^2)\eta_i = 0 }$ propagates on water of depth $H_0$ towards an area $\Omega \subset \R^2$.  It is then scattered by the variable topography (represented by the depth map $d(X)$) and the total wave field $\eta = \eta_{i} + \eta_{s}$, where $\eta_{s}$ is the scattered wave, satisfies 
$$ \nabla \cdot c \nabla \eta + k^2c \eta = 0 \quad \text{in } \R^2,$$

The example we have in mind is ocean waves propagating towards an area $\Omega$ of shallow water and variable topography, e.g., near the coast, but the model is not limited to this setting. Figures \ref{fig: 3dwave} and \ref{fig: depthandwaves} illustrate the situation, and the plotted waves are produced by numerically solving the scattering problem for the MSE. 

\begin{figure}
    \centering
    \includegraphics[width = 1.1\textwidth]{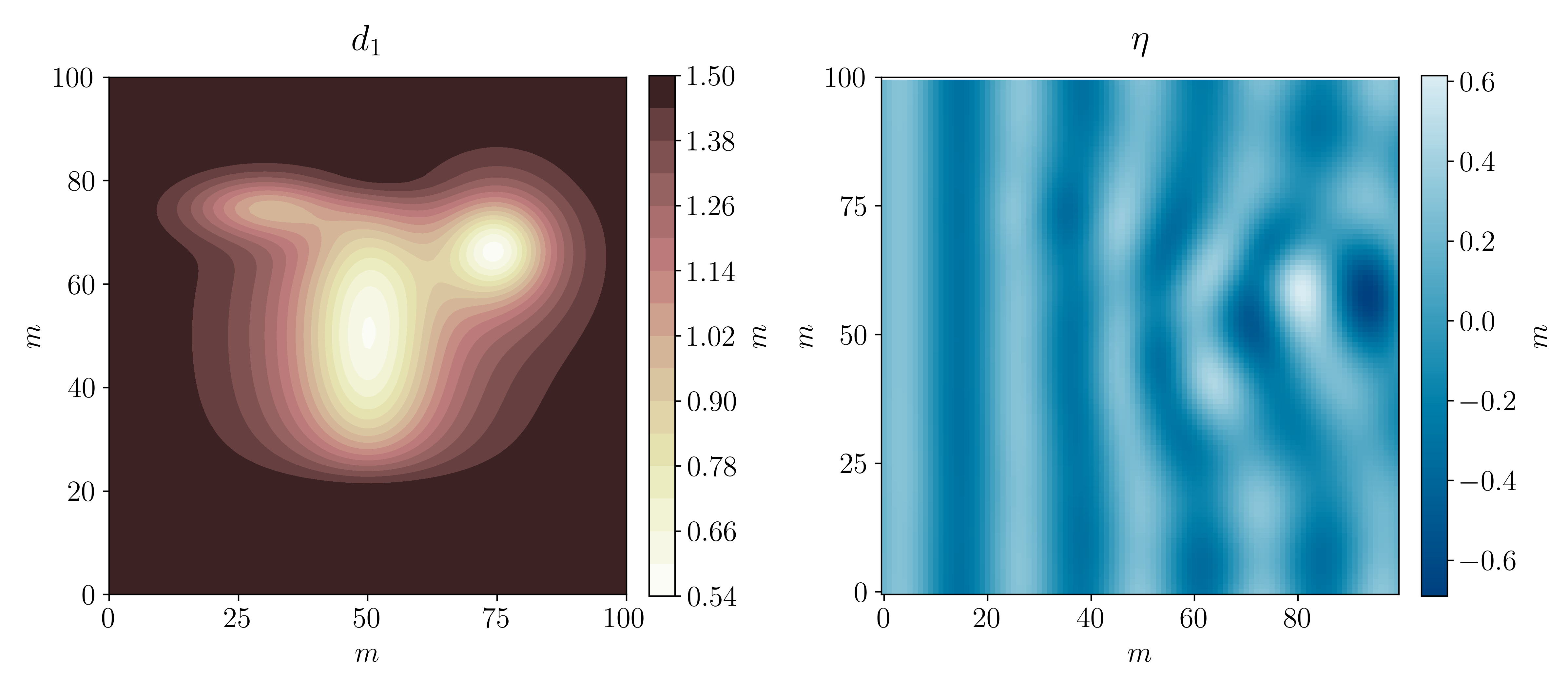}
    \caption{To the left, we see a contour plot of the depth function $d(X)$ on a $100 \times 100$ meter domain. To the right, we see the water wave $\text{Re}\{\eta(X)\}$ coming in from the left and being distorted as it travels past the variable depth. The angular frequency of the wave is $\omega = 1$, and the computation is done using the method described in Section \ref{section: duan-rokhlin} and the simplified mild-slope equation from Section \ref{section: simplification}.}
    \label{fig: depthandwaves}
\end{figure}

To guarantee a unique, outward propagating scattered wave, we require $\eta_{s}$ to satisfy the Sommerfeld radiation condition (cf. \cite{coltonkress,pike2001scattering}, or Chapter 1 in \cite{kuznet2002linear} for a physical explanation). In summary, we now have the following scattering problem: Let $\Omega \subset \R^2$ be a smooth, compact domain and $d(X) = H_0 - h(X)$ be the topography.  Assume that $h \in C^2_0(\R^2)$ with $\text{supp}(h) \subset \Omega$, that $\min d(X) = H_0 - h(X) > 0  $, and that $d(X)$ satisfies the mild-slope condition. For an incoming wave ${\eta_{i} = a\exp(i\bm{k}_i\cdot X)}$, find the scattered wave $\eta_{s}$ such that 

\begin{equation}
    \begin{cases}
        &\nabla \cdot c \nabla \eta + k^2 c \eta = 0, \quad \text{in} \quad  \R^2, \\
        &\eta = \eta_{i} + \eta_{s},\quad \text{in} \quad \R^2,  \ \\
        &\lim_{|X| \to \infty } \sqrt{|X|}\left( \partial_{|X|} - ik_i\right) \eta_{s} = 0, \quad \text{uniformly for } X/|X| \in S^1. 
    \end{cases}
    \label{eq: scattering}
\end{equation}

The scattering problem for water waves now very much resembles the classical acoustic scattering problem, see, e.g., Chapter 8 in \cite{coltonkress} (3D) or Chapter 1 on \cite{pike2001scattering} (2D), except for the appearance of the coefficient $c$ and the implicit dependence on the topography through the coefficients. By utilizing the Liouville transformation, the well-posedness of the scattering problem is therefore quickly established from existing theory.

\begin{proposition}
    For an incoming wave  $ \eta_{i} = a\exp(i\bm{k}_i\cdot X)$ there exists a unique solution ${\eta_s \in C(\R^2)}$ to the scattering problem \eqref{eq: scattering}. The total field $\eta$ is given by $$\eta = c^{-1/2} \tilde{\eta},$$ where $\tilde{\eta}$ is the solution to 
    \begin{equation}
        \tilde{\eta}(X) = \eta_i(X) - k_i^2 \int_{\R^2} G_{k_i}(X,Y)  q(Y) \tilde{\eta}(Y) \mathrm{d}Y. 
    \end{equation}
    Here, $G_{k_i}(X,Y) = \frac{i}{4}H_0^{(1)}(k_i|X-Y|)$ is the Hankel function of the first kind and order 0, and 
    $$ q = 1 - \frac{1}{k_i^2}\left(k^2 - \frac{\Delta c^{1/2}}{c^{1/2}}\right).$$
    \label{prop: wellposed}
\end{proposition}

\begin{proof}
First, note that $c(X) > 0 $ and  $c \in C^2$. We apply the Liouville transformation to $\eta$: setting $\tilde{\eta} = c^{1/2}\eta$, one finds that $\tilde{\eta}$ satisfies
\begin{equation}
     \Delta \tilde{\eta} + \tilde{k}^2 \tilde{\eta} = 0,
\end{equation}
where $\tilde{k}^2= k^2 - \frac{\Delta c^{1/2} }{c^{1/2}}$. As $c = 1$ outside $\Omega$, the radiation condition is unaltered, and the scattering problem can now be written 
\begin{equation}
    \begin{cases}
        &\Delta \tilde{\eta} + \tilde{k}^2  \tilde{\eta} = 0, \quad \text{in} \quad  \R^2, \\
        &\tilde{\eta} = \eta_{i} + \tilde{\eta}_{s},\quad \text{in} \quad \R^2,  \ \\
        &\lim_{|X| \to \infty } \sqrt{|X|}\left( \partial_{|X|} - ik_i\right) \tilde{\eta}_{s} = 0, \quad \text{uniformly for } X/|X| \in S^1. 
    \end{cases}
    \label{eq: liouville scattering}
\end{equation}
The fundamental solution satisfying $(\Delta + k_i^2)G_{k_i}(X,Y) = -\delta (|X-Y|)$ and the radiation condition is $ G_{k_i}(X,Y) = \frac{i}{4}H_0^{(1)}(k_i|X-Y|)$ (Section 1.2 in \cite{pike2001scattering}), and by writing 
$$ \Delta \tilde{\eta} + k_i^2  \tilde{\eta} = k_i^2\left(1 - \frac{\tilde{k}^2}{k_i^2}\right) \tilde{\eta}, $$
and using that $(\Delta + k_i^2)\eta_i = 0$, we obtain the so-called Lippmann-Schwinger formulation of \eqref{eq: liouville scattering} (cf. \cite{coltonkress} for a full proof of the equivalence of the PDE and Lippmann-Schwinger formulation): 
\begin{equation}
    \tilde{\eta}(X) = \eta_i(X) - k_i^2 \int_{\R^2} G_{k_i}(X,Y)  q(Y) \tilde{\eta}(Y) \mathrm{d}Y, \quad X \in \R^2.
    \label{eq: lippmann schwinger}
\end{equation}
Since $q(Y) = 0  $ for $Y \notin \Omega$, we define the integral operator 
\begin{equation}
    V u (X) =  k_i^2 \int_{\Omega} G_{k_i}(X,Y)  q(Y) u(Y) \mathrm{d}Y, \quad X \in \R^2. 
    \label{eq: V}
\end{equation}
From \cite{pike2001scattering}, the fundamental solution $G_{k_i}(X,Y)$ has the asymptotic behaviour 
\begin{equation}
    G_{k_i}(X,Y) \sim \begin{cases}
        i\sqrt{\frac{1}{8\pi r}}\exp(ik_i(r - \frac{\pi}{4})\left(1 + \mathcal{O}(1/r)\right), \quad r = |X-Y| \to \infty, \\
        \ -\frac{1}{2\pi}\log(r) + C + \mathcal{O}(r^2\log(r)), \quad r = |X - Y| \to 0. 
    \end{cases}
    \label{eq: hankel exp}
\end{equation}
Hence, $G_{k_i}(X,Y)$ is weakly singular, and it follows that for any compact set $ \mathcal{V} \subset \R^2$ containing $\Omega$,  ${V : C(\mathcal{V}) \to C(\mathcal{V})}$ is compact (Chapter 2, \cite{kress1989linear}).
Writing \eqref{eq: lippmann schwinger} as 
\begin{equation}
    (I + V)\tilde{\eta} = \eta_i, \quad (I + V) : C(\mathcal{V}) \to C(\mathcal{V}),
\end{equation}
we can apply Fredholm theory and conclude that the above equation has a unique solution $\tilde{\eta} \in C(\mathcal{V})$ that depends continuously on $\eta_i$ if and only if the only solution to ${(I + V) \tilde{\eta}} = 0$ is $\tilde{\eta} = 0$ (Chapter 4, \cite{kress1989linear}). This is indeed the case, and follows from the uniqueness of the Helmholtz equation with radiation condition, and proofs of this can be found in \cite{pike2001scattering,coltonkress,hsiao2008boundary,hubert2012vibration}. Last, when $\tilde{\eta}$ is found, the extension of the solution outside $\mathcal{V}$ is given by \eqref{eq: lippmann schwinger}.  
\end{proof}

\subsection{A simplification of the mild-slope equation}
\label{section: simplification}
From the considerations in Section \ref{section: depth dependence}, it is clear that not much is gained towards the goal of reconstructing the topography $d(X)$ from knowing both $c$ and $k$. Still, the two coefficients appear in a very entangled way, and so it is not evident from \eqref{eq: mse} if it is possible to reconstruct only $k$. However, the special form of the MSE allows for a simplification. Expanding \eqref{eq: mse} and dividing by $c$ we have 
\begin{equation*}
    \Delta \eta + \frac{\nabla c}{c} \cdot \nabla \eta + k^2 \eta = 0.
\end{equation*}
If the term $\frac{\nabla c}{c} \cdot \nabla \eta$ is negligible, we can instead work with the simplified equation 
$\Delta \zeta + k^2 \zeta = 0$, that is, equation \eqref{eq: scattering} with $c = 1$. The following proposition says that under the assumptions of the mild-slope equation, the simplification is valid when the maximal wavenumber $k$ is not too large and when the term $|\Delta d|/d$ is not very large, i.e., when the bottom does not have small rapid oscillations near the surface.   

\begin{proposition}
    Assume the depth map $d$ satisfies the mild-slope condition \eqref{eq: mild slope condition}. For an incoming wave $ \eta_i = a \exp(\bm{k}_i \cdot X)$ with $ a k_i \leq \varepsilon \ll 1 $, let $\eta$ be the solution to \eqref{eq: scattering} and let $\zeta$ be the solution to \eqref{eq: scattering} with the coefficient 
    $c = 1$. 
    Then $$ \|\eta - \zeta \|_\infty \leq C \varepsilon \delta  \|k\|_\infty\left(k_i +\frac{(1+4\delta)\|k\|_\infty^2}{k_i} + \frac{\|\Delta d/d  \|_\infty}{k_i}\right),$$
    where $\delta \ll 1$ is the mild-slope coefficient and $C$ is some constant depending $\Omega$. 
    \label{prop: simplification}
\end{proposition}

\begin{proof}

In the Appendix, the following estimates are derived: 
\begin{equation}
    \frac{|\nabla c|}{c} \leq  \frac{5|\nabla d|}{4d} \quad \text{and} \quad k_i^2\|q\|_\infty \leq k_i^2 +(1+4\delta)\|k\|_\infty^2 + \|\Delta d/d  \|_\infty.
    \label{eq: cq estimates }
\end{equation}
We now estimate $\nabla \eta$. From \eqref{eq: lippmann schwinger}, we have
\begin{equation}
    \nabla \eta = \nabla \eta_i - k^2_i \int_{\Omega} \nabla G_{k_i}(X,Y) q(Y)\eta(Y) \mathrm{d}Y.
\end{equation}
By equation \eqref{eq: hankel exp}, we have that $\nabla G_{k_i}(X,Y) \sim 1/|X -Y|$ as $|X -Y| \to 0 $, and so 
$$ \int_{\Omega} \partial_{x_i} G_{k_i}(X,Y) q(Y)\eta(Y) \mathrm{d}Y : C(\mathcal{V}) \to C(\mathcal{V}) $$
is weakly singular on $\R^2$ and hence bounded (Chapter 2, \cite{kress1989linear}). As $\eta$ depends continuously on $\eta_i$  follows that $$\|\nabla \eta\|_\infty \leq C k_i^2 \|\eta_i\|_\infty \|q\|_\infty $$ for some constant $C$ depending on $\Omega$.  
We now set $\varphi = \zeta -\eta$, and find that $\varphi$ satisfies 
$$ \Delta \varphi + k_i^2 \varphi = k_i^2(1 - k^2/k_i^2)\varphi + \frac{\nabla c \cdot \nabla \eta}{c}.$$
Using \eqref{eq: V}, we write

$$ (I + V)\varphi = f, \quad f = \int_\Omega G_{k_i}(X,Y)\left(\frac{\nabla c \cdot \nabla \eta}{c}\right)(Y) \mathrm{d}Y.$$
By \eqref{eq: cq estimates } and the mild-slope condition \eqref{eq: mild slope condition}, we have that $\|\nabla c/c\|_\infty \leq \frac{5}{4}\delta \|k\|_\infty$, and so 
$$\|\varphi\|_\infty \leq C \|f \|_\infty \leq C k_i^2 \|\tilde{\eta}_i\|_\infty\|q\|_\infty \frac{5}{4}\delta \|k\|_\infty. $$
With $\|\eta_i\|_\infty = a$ we get
\begin{align*}
    \|\eta(X) - \zeta(X)\|_\infty \leq C \varepsilon \delta  \|k\|_\infty\left(k_i +\frac{(1+4\delta)\|k\|_\infty^2}{k_i} + \frac{\|\Delta d/d  \|_\infty}{k_i}\right).
\end{align*}
\end{proof}

\begin{remark}
    Ocean waves are relatively long waves, and so their wave numbers are small. In (\cite{holthuijsen2010waves}, Chapter 6), it is shown that for wind-generated ocean waves in the linear regime in deep water, the average wave period is $T \approx 2\pi$, and this also agrees well with experimental findings. Consequently, $\omega = 2\pi/T = 1$, and so the corresponding wavenumber for the incoming wave is 
    $k_i = \frac{\omega^2}{g} \approx 1/10 $, and if, say, $\min (d(X)) = 0.5$, then $\|k\|_\infty \approx 0.2$. It thus seems like a reasonable choice to use the simplified version of the MSE with $c=1$. 
\end{remark}

\subsection{Wave measurements}

The surface wave field $\zeta$ associated with $\eta$ is of course not complex-valued, but given by 
$$\zeta(X,t) = \text{Re}\left\{\eta(X)\exp(i\omega t)\right\} = \text{Re}\left\{\eta(X)\right\}\cos(\omega t) - \text{Im}\left\{\eta(X)\right\}\sin(\omega t).$$
Hence, observing $\zeta(X,t)$ at a single instant does not determine the spatial wave pattern $\eta(X)$. However, observing it at two instants $t_1$ and $t_2$, we have 

\begin{equation*}
    \begin{bmatrix}
        \zeta(X,t_1) \\
        \zeta(X,t_2)
    \end{bmatrix}
    = \begin{bmatrix}
        \cos(\omega t_1) \hspace{1mm} -\sin(\omega t_1) \\
        \cos(\omega t_2) \hspace{1mm} -\sin(\omega t_2)
    \end{bmatrix}\begin{bmatrix}
        \text{Re}\left\{\eta(X)\right\} \\
        \text{Im}\left\{\eta(X)\right\}
    \end{bmatrix}
\end{equation*}
and the matrix is invertible when $\sin(\omega t_1) \neq 0$ and $t_2 \neq t_1 + \frac{\pi n}{\omega}, n \in \mathbb{Z}$.  
If one observes $\zeta(X,t)$ for $t \in [0,T]$ for $T$ large enough, one can also compute $\eta(X)$ by the Fourier transform.
This is beneficial when the total wave field is a superposition of waves with different angular frequencies and then
\begin{equation*}
    \eta(X) = \mathcal{F}\left(\sum_{\omega_n \in \Lambda} \zeta(X,t;\omega_n) \right)(\omega)
\end{equation*}
where $\Lambda$ is a set of angular frequencies and $\zeta(X,t;\omega_n) = \text{Re}\left\{\eta(X;\omega_n)\exp(i\omega_n t)\right\}$ were for each $\omega_n \in \Lambda$, $\eta(X;\omega_n)$ is the solution to \eqref{eq: scattering} with $\omega = \omega_n$ and some incoming wave $\eta_{i,n}$.  

\subsubsection*{Remote sensing of the wave field}
There are several methods to measure the wave amplitude $\eta$. Using only images or video, one can use the pixel intensity and the angle of the incoming sunlight to estimate $\eta$ \cite{almar2021sea}. By using stereo vision techniques, where two or more video cameras record the same scene, it is possible to use feature extraction and triangulation to estimate the surface elevation $\zeta(X,t)$ \cite{holthuijsen1983stereophotography,wanek2006automated,bergamasco2017wass}. In the paper \cite{kabel2019mapping}, a new LiDAR-based method is investigated for real-time observation of the pointwise wave height. We will assume that we have one of these methods at hand and that we can measure $\eta(X)$ on the domain with variable topography. \\

 We define the wave measurements as follows. Let  $\mathcal{D}$ be a domain containing $\Omega = \text{supp}(h)$ and let $\{X_i\}_{i = 1}^N \subset \mathcal{D}$ be a discrete set of points in $\mathcal{D}$. Let $\varepsilon \in L^2(\mathcal{D})\cap L^\infty(\mathcal{D})$ be an unknown function representing the noise and inaccuracies in the measurement of $\eta(X)$. We set
    \begin{align*}\mathcal{M} &= \eta(X)|_\mathcal{D} , \\   
    \mathcal{M}_\varepsilon &=  \eta(X)|_\mathcal{D} + \varepsilon , \\
    \quad \mathcal{M}_\varepsilon^d &= \{ \eta(X_i) + \varepsilon(X_i) \}_{i = 1}^N. 
    \end{align*}
We now have in place what we need to consider the inverse problem.

\section{The inverse problem}
\label{sect: inversion}
Based on Proposition \ref{prop: simplification}, we work with the simplified MSE, i.e., we assume our wave satisfies equation \eqref{eq: scattering} with $c = 1$. The first proposition is a uniqueness result for the determination of the topography in the case of a perfect, continuous measurement $\mathcal{M}$. 

\begin{proposition}
    The measurement $\mathcal{M}$ uniquely determines the topography $d(X)$. In fact, knowing $\eta(X)$ on any open set $\mathcal{O} \subset \R^2$ uniquely determines $d(X)$. 
    \label{prop: uniqueness}
\end{proposition}

Although the uniqueness is important and interesting in itself, the question of stability is the most important for applications. As any real measurement will suffer from noise and other inaccuracies, it is crucial to know that one is still able to reconstruct useful information about the quantity of interest. However, as we saw in Section \ref{section: depth dependence}, one cannot hope to accurately reconstruct $d(X)$ from the knowledge of a noisy wavenumber $k_\varepsilon(X)$ when $ k_\varepsilon(X) \approx \mu = \omega^2/g$. We therefore propose the truncated depth function $d_\alpha$:
\begin{definition}
    For $\alpha > 0$ and $\mu = \omega^2/g$, define $d_\alpha$ by 
    \begin{equation*}
        d_\alpha(k) = \begin{cases}
            \frac{\tanh^{-1}(\mu/k))}{k}, \quad \text{for } k \geq \mu + \alpha, \\
           \frac{\tanh^{-1}(\mu/(\mu +\alpha))}{\mu + \alpha}, \quad \text{for } k < \mu + \alpha.
        \end{cases}
    \end{equation*}
    \label{def: dalpha}
\end{definition}
Truncating $d(k)$ at $k = \mu + \alpha$, we discard wavenumbers too close to the critical value $\mu$, and aim only to reconstruct  $d(X)$ up to the maximal depth $d_\text{max} = d(\mu + \alpha)$. If $k(d_\text{max}) \geq \mu + \alpha$, then $d_\alpha(k) = d(k)$.

As $d_\alpha$ has Lipschitz constant $M_{\alpha, \mu} \approx \frac{1}{\alpha^2}$, it is clear that if $\alpha$ is large enough, reconstruction of $d_\alpha$ from $k_\varepsilon$ is stable, but the price we pay is that we get no information about $d(X)$ when $k(X) < \mu + \alpha$. One can, of course, set $d_\alpha(k) = H_0$ (or any other value) when $k < \mu + \alpha$. 

\begin{remark}
For a given truncation parameter $\alpha$ and frequency $\omega$, we can estimate the maximal depth that will be reconstructed. By using the approximation\footnote{Numerical investigation shows that the maximal error in $k$ using this approximation, for $d \in [0.2,100]$ and $\omega \in [0.5,1.5]$, is $\epsilon \approx 0.015$, and this occurs for $\omega =1.5$.} $\tanh(x) \approx x/(1 + x^2)^{1/2}$ (See \cite{bagul2021tight}), we find from \eqref{eq: dispersion} that 
$$ k \approx \frac{\sqrt{\mu^2 + \frac{\mu\sqrt{\mu^2d^2 + 4}}{d}}}{\sqrt{2}}, $$
and requiring $k \geq \mu + \alpha$, we find that 
$$ d_{\text{max}}  \approx \frac{\mu}{\sqrt{\alpha}\sqrt{\alpha^3 + 4\alpha^2\mu + 5\alpha\mu^2 + 2\mu^3 }}.$$
Moreover, if $d_{\text{max}}$ is known, then one can choose $\alpha $ such that 
$$\alpha = \frac{\sqrt{\mu^2 + \frac{\mu\sqrt{\mu^2d_{\text{max}}^2 + 4}}{d_{\text{max}}}}}{\sqrt{2}} - \mu.$$
\end{remark}

Another fundamental problem for the inverse problem is that for subsets $\mathcal{N} \subset \mathcal{D}$ such that $ \mathcal{M}_\varepsilon|_{\mathcal{N}} \approx 0$, i.e., areas where the \emph{noisy} measurement is approximately zero, we cannot hope to discern any usefull information about $k(X)$. The solution is to partition the domain into two sets: one where the measurement is of sufficient magnitude, and we can hope to reconstruct the depth, and one where we cannot. However, as the underlying PDE in our problem is linear, any constant multiple of the solution will satisfy the equation. Hence, introducing sets where the measurement is too small needs to be unambiguous with respect to scaling by a constant.  We therefore define the following subset of $\mathcal{D}$:
\begin{definition}
    For $0< \gamma \leq 1$, let $\mathcal{D}_\gamma = \left\{ X \in \mathcal{D} : \frac{|\mathcal{M}_\varepsilon(X)|}{\|\mathcal{M}_\varepsilon \|_\infty} \geq \gamma \right\}$.
    \label{def: D_gamma}
\end{definition}
Hence, for a given $\gamma$, the set $\mathcal{D}_\gamma $ is the set on which we will attempt to reconstruct the topography. 

We are now ready to state our main result. 

\begin{proposition}
    For any $\alpha,$ $0 <\gamma \leq  1$, let $d_\alpha^\varepsilon$ be the reconstruction of $d_\alpha$ from a noisy measurement $\mathcal{M}_\varepsilon$. 
    Then it holds that 
    $$\|d_\alpha^\varepsilon - d_\alpha \|_{L^2(\mathcal{D}_\gamma)} \leq C \gamma^{-1/2}\left(\frac{3k_i}{\alpha(\alpha + \mu)}\right)\sqrt{ \|\mathcal{M}_\varepsilon - \mathcal{M} \|_{L^2(\mathcal{D})}},$$
    where the constant $C$ satisfies     $$C =  \left(|\mathcal{D}|^{1/2}\left( \|q\|_\infty  + \tilde{C} \sqrt{\|q\eta\|_{H^2(\mathcal{D})}/k_i^2} \right)\right)^{1/2},
    $$
    and $|\mathcal{D}|$ denotes the area of the measurement domain. 
    \label{prop: stability}
\end{proposition}
\begin{remark} 
\label{remark: stability}
The above result expresses the main instability mechanism of the inverse problem. If we want to either reconstruct the topography in a potentially deep area or reconstruct the topography from waves with a high wavenumber, we need to choose $\alpha $ very small, as then $\tanh(kd) \approx 1$ and $k \approx \mu$. But choosing $\alpha$ small results in a large constant $C$ and increasing uncertainty in the reconstructed topography. Similarly (but less fatal in practice), for the choice of a lower bound $\gamma$, i.e., the threshold for which we say that we that a noisy measurement is too small and yields no useful information, we see that when $\gamma \to 0$ the constant blows up. On the other hand, we can choose $\alpha$ or $\gamma$ (or both) large; this reduces the constant and hence increases stability but limits the depth we can reconstruct or the region $\mathcal{D}_\gamma$ where we can reconstruct the topography. 
In any case, we can conclude that longer waves should result in a more stable reconstruction of the topography but that if the water is too deep, reconstruction will be unstable.    
\end{remark}

To prove the above propositions, we consider the inversion in two steps: 
\begin{itemize}
    \item[1:] Given a measurement $\mathcal{M}$ or $\mathcal{M}_\varepsilon$, reconstruct $q(X) = 1 - k(X)^2/k_i^2$. 
    \item[2:] Reconstruct the truncated depth function $d_\alpha(X)$ from $q(X)$.
\end{itemize}

For the reconstruction of $q$ we proceed as follows: rearranging \eqref{eq: lippmann schwinger}, we have that 
\begin{equation*}
    \mathcal{M}(X) - \eta_i(X) = k_i^2\int_\mathcal{D} G_{k_i}(X,Y) \mathcal{M}(Y) q(Y) \mathrm{d}Y, \quad X \in \mathcal{D}.
\end{equation*}
Writing 
\begin{equation}
    b_\varepsilon = \mathcal{M}_\varepsilon -  \eta_i  \quad \text{and} \quad  Ku = \int_\mathcal{D} G_{k_i}(X,Y) u(Y) \mathrm{d}Y, 
    \label{eq: bandK}
\end{equation}
we consider the problem 
\begin{equation}
    \text{Find } u \text{ such that} \quad k_i^2Ku = b_\varepsilon.
    \label{eq: fredholm}
\end{equation}
As $K$ is a compact operator, solving this equation for $u$ is a Fredholm equation of the first kind and the classical example of an ill-posed problem. If we can solve it, we get in the case of exact measurement $\mathcal{M}$ that $ q = \frac{u}{\mathcal{M}}$, and consequently $k(X) = k_i\sqrt{1- q(X)}$. For noisy measurements, however, both the inversion of $K$ and division by $\mathcal{M}_\varepsilon$ require further investigation. The following properties show that the problem of finding $u$ is only mildly ill-posed and that one can obtain an optimal stability estimate by using Tikhonov regularization. 

\begin{proposition}
    The operator $K : L^2(\mathcal{D}) \to L^2(\mathcal{D})$ is compact and injective, and there are constants $0 < c \leq C$ such that the singular values of $K$ satisfy   
    $$c n^{-1} \leq  s_n(K) \leq C n^{-1} \quad \text{for } n = 1,2,3,... \hspace{1mm}.$$ 
    \label{prop: K}
\end{proposition}
The degree of ill-posedness can be quantified by the decay of the singular values of the operator, and the above result shows that $ s_n(K) \sim n^{-1}$. Hence, the decay is not too severe, and the inverse problem of finding $u$ from $Ku = b$ is then said to be \emph{mildly ill-posed} (Chapter 3, \cite{mueller2012linear}).

To find an approximate solution to \eqref{eq: fredholm} from noisy data, we consider the classical Tikhonov regularization. Let $\lambda > 0$ be the regularization parameter, and let $u_\delta^\lambda$ be the minimizer of the Tikhonov functional, i.e.,  
\begin{equation}
    u_\delta^\lambda = \text{argmin}_{u \in L^2(\mathcal{D})} \| b_\varepsilon - k_i^2K u \|_{L^2(\mathcal{D})}^2 + \lambda \|u\|_{L^2(\mathcal{D})}^2,
    \label{eq: tikhonov}
\end{equation}
where $b_\varepsilon = \mathcal{M}_\varepsilon - \eta_i$ (See e.g. \cite{hanke2017taste}). 
We then have the following result:
\begin{proposition}
    Let $\delta > 0$ and assume $\|\mathcal{M}_\varepsilon - \mathcal{M} \|_{L^2(\mathcal{D})} \leq \delta$. Then there exists $\lambda > 0$ such that $$ \|u_\delta^\lambda - u\|_{L^2(\mathcal{D})} \leq C \sqrt{\|u\|_{H^2(\mathcal{D})}} \frac{\sqrt{\delta}}{k_i}$$ for some constant $C > 0$, where $u_\delta^\lambda$ is the solution to \eqref{eq: tikhonov}.  
    \label{prop: tikhonov}
\end{proposition}

The above optimal accuracy holds as $u$ in \eqref{eq: fredholm} is shown to satisfy a so-called source condition. It shows that the error in the regularized solution $u_\delta^\lambda$ is bounded by the measurement error and is the main result leading to Proposition \ref{prop: stability}. As the proofs of these propositions are a bit lengthy, we have moved them to the Appendix. 

\section{Reconstruction method and numerical experiments}
\label{sect: numerical}

In this section, we investigate the inverse problem numerically. We consider a discrete, noisy measurement $\mathcal{M}^d_\varepsilon$, and test the proposed inversion method for two different topographies and two different incoming waves. We first describe a simple, regularized inversion procedure.

\subsection{A simple reconstruction method}
\label{Section: discrete inversion}
The formulation used to prove the stability properties in Section \ref{sect: inversion} can, in principle, be used to solve the inverse problem. However, the fact that we need to subtract the incoming wave from the measurement and invert a possibly very large matrix that depends on the possibly uncertain wavenumber $k_i$ of the incoming wave makes it less than ideal. We instead propose a very simple method based on the following observation: As the wave field $\eta$ is assumed to satisfy the simplified MSE, we have that 
\begin{equation}
-\Delta \eta = k^2\eta \quad \text{in} \quad  \mathcal{D}. 
\label{eq: poisson}
\end{equation} 
For $\varepsilon > 0$, let the set $\Omega_\varepsilon$ be such that 
 $\Omega \subset \Omega_\varepsilon \subset \mathcal{D}$ and  $$\min(\text{dist}(\partial \Omega_\varepsilon,\Omega),\text{dist}(\partial \Omega_\varepsilon,\mathcal{D})) > \varepsilon.$$ Let $\varphi \in C^\infty_0(\R^2)$ be a smooth cut-off function  such that $\varphi|_{\Omega} = 1$ and $\varphi|_{\R^2\setminus \Omega_\varepsilon} = 0$. Then 
 $$ -\Delta (\varphi \eta) = u $$
where $u = \varphi k^2 \eta -\nabla \varphi \cdot \nabla \eta - \eta \Delta \varphi$ and so $u|_\Omega = \eta k^2$. With $\eta$ replaced by $\mathcal{M}_\varepsilon$, we need to apply the Laplacian to the noisy measurement $\mathcal{M}_\varepsilon$ to estimate $u = \eta k^2$. Differentiating noisy data is, of course, an ill-posed problem, and so we need a regularization method. Recalling that true wave $\eta$ should have a wavenumber $k \approx k_i$, and that, by assumption, $k_i$ is small, a natural idea is to apply a low-pass filter to $\mathcal{M}_\varepsilon$ before we apply $-\Delta$. More precisely, let $\mathcal{G}_\sigma$ be a Gaussian low-pass filter given by 
$$ \mathcal{G}_\sigma \mathcal{M}_\varepsilon(X) = \int_\mathcal{D} g_\sigma(X,Y)  \varphi(Y)\mathcal{M}_\varepsilon(Y) \mathrm{d}Y, \quad g_\sigma(X,Y) = \frac{1}{2\pi\sigma^2}\exp\left(-\dfrac{|X-Y|^2}{2\sigma^2}\right).$$
We now have that 
$$ \Delta (\mathcal{G}_\sigma \mathcal{M}_\varepsilon)(X) = \int_\mathcal{D}  \Delta_X g_\sigma(X,Y)  \varphi(Y)\mathcal{M}_\varepsilon(Y) \mathrm{d}Y. $$

Assuming now that the measurement domain is rectangular, $\mathcal{D} = (0,L)\times (0,L)$, we denote by $\{v_n\}_{n\in \mathbb{Z}^2}$ the usual Fourier basis, i.e., ${v_n = L^{-1}\exp(2\pi i(n_1x_1 + n_2x_2)/L)}$, and let $$\mathcal{F}(g)[n] = \int_\mathcal{D}\overline{v}_n(Y)g(Y) \mathrm{d}Y, \quad n \in \mathbb{Z}^2.$$
If we assume that $3\sigma \leq L/2 $ (i.e., the Gaussian is essentially supported in $\mathcal{D}$), we have  
$$  \mathcal{F}((\Delta g_\sigma)((L/2,L/2),\cdot))[n] \approx L^{-3}|n|^2\exp\left(-2\pi^2  \sigma^2|n|^2/L^2\right), $$
and so, by the convolution property of Fourier series, 
\begin{align*}
    \mathcal{F}\left(\Delta (\mathcal{G}_\sigma \mathcal{M}_\varepsilon)\right)[n] &\approx \mathcal{F}((\Delta_X g_\sigma)((L/2,L/2),\cdot))[n]\mathcal{F}(\varphi\mathcal{M}_\varepsilon)[n] \\
    & = L^{-3}|n|^2\exp\left(-2\pi^2  \sigma^2|n|^2/L^2\right)\mathcal{F}(\varphi\mathcal{M}_\varepsilon)[n] , \quad n \in \mathbb{Z}^2.
\end{align*}
Hence, we obtain a regularized reconstruction  $U_\sigma $ of $u|_\Omega =  \eta k^2$ by 
\begin{equation}
     U_\sigma(X) = \mathcal{F}^{-1}\left(L^{-3}|\cdot|^2\exp\left(-2\pi^2 \sigma^2|\cdot|^2 /L^2\right)\mathcal{F}(\varphi\mathcal{M}_\varepsilon)[\cdot]\right)(X), \quad X \in \Omega.
     \label{eq: U filtered}
\end{equation}
For the second step of the inversion, we use a simple, regularized division. Let $\mathcal{M}^\sigma_\varepsilon = \mathcal{G}_\sigma \mathcal{M}_\varepsilon,$ and for $\gamma > 0$, set 
\begin{equation}
    k_\sigma^2(X) = \frac{|U_\sigma(X)|}{|\mathcal{M}^\sigma_\varepsilon(X)| + \gamma}, \quad X \in \Omega. 
\end{equation}
As a heuristic for choosing $\sigma$, we use the assumption that the magnitude of the wavenumber of $\eta$ should be close to $k_i$. Hence, we do not want to suppress too much the frequencies $k_n = 2\pi|n|^2/L^2$ in our Fourier representation when $k_n \approx k_i$. Based on \eqref{eq: U filtered},  we therefore choose $\sigma$ such that $\sigma^2 \pi k_i^2 < 1$. 
Once $k_\sigma^2$ is computed, we choose a suitable $\alpha > 0$ and reconstruct $d_\alpha^\varepsilon$ from $k_\varepsilon$. 
\subsubsection*{Discrete reconstruction}
For simplicity, we now assume that the  discrete, noisy measurement $\mathcal{M}_\varepsilon^d$ is uniformly sampled on domain $\mathcal{D}^d = (nL/N,mL/N)_{n,m = 0}^{N-1}$. By the above considerations, we then suggest the following reconstruction method: 
\begin{itemize}
    \item[1:] Given the incoming wavenumber $k_i$, choose $\sigma$ such that $k_i^2 \pi  \sigma^2 < 1$
    \item[2:] Let $g_\sigma(L/2,L/2,\mathcal{D}^d)$ and $\Delta g_\sigma(L/2,L/2,\mathcal{D}^d)$ denote the matrices consisting the values of the Gaussian kernel and its Laplacian, respectively, centered in the middle of $\mathcal{D}$ and evaluated at $\mathcal{D}^d$. We denote their discrete Fourier transforms (DFTs) by $\hat{G}^d$ and $\hat{\Delta G^d}$, respectively. 
    \item[3:] Compute the DFT $\hat{\mathcal{M}}_\varepsilon^d$ of the measurement $\mathcal{M}_\varepsilon^d$. Compute
    \begin{align*}
    U_\sigma^d &= \mathcal{F}^{-1}_{\text{DFT}}\left( \hat{\mathcal{M}}_\varepsilon^d \circ \hat{\Delta G^d}\right), \\
    \mathcal{M}_\sigma^d &= \mathcal{F}^{-1}_{\text{DFT}}\left( \hat{\mathcal{M}}_\varepsilon^d \circ \hat{G^d}\right),
    \end{align*}
    where $\mathcal{F}^{-1}_{\text{DFT}}$ is the inverse DFT and $\circ$ denotes elementwise matrix multiplication. 
    \item[4:] Choose $\gamma > 0$ and $\alpha > 0 $ and compute 
    $$k_\varepsilon = \sqrt{\frac{|U_\sigma^d|}{|\mathcal{M}_\sigma^d| + \gamma }} \quad \text{and} \quad d_\alpha(k_\varepsilon).$$
    The choice of $\gamma$ and $\alpha$ should be based on apriori knowledge of the expected depth, the properties of the incoming wave and the noise level, as outlined in Remark \ref{remark: stability}. 
\end{itemize}
To test the proposed method, we must compute the solutions to the mild-slope equation. 

\subsection{Numerical solution of the forward problem}
\label{section: duan-rokhlin}
We discretize the weakly singular integral operator 
$$(Kq)u(X) = \int_\mathcal{D} G_{k_i}(X,Y)q(Y)u(Y) \mathrm{d}Y, \quad X \in \mathcal{D}$$
using the $4$'th order Duan-Rokhlin quadrature method \cite{duan2009high}. For $L > 0$ and $N \in \mathbb{N}$, let $\mathcal{D} = (0,L)\times (0,L)$ and let $\mathcal{D}^d = \{X_n\}_{n = 0}^{(N-1)^2}$ be a uniform discretization of $\mathcal{D}$, with $X_n = (i(n)h,j(n)h)$, suitable index functions $i(n),j(n)$ and step size $h = L/N$. 
We define the discrete operator $(Kq)^d \in \mathbb{C}^{N^2\times N^2} $ as
\begin{equation*}
    (Kq)^d_{m,n} = h^2 g(m,n,q) \quad \text{for} \quad  0 \leq m,n \leq (N-1)^2,
\end{equation*}
where 
\begin{equation*}
     g(m,n,q) = \begin{cases} \frac{i}{4}H^{(1)}_0(k_i|X_m - X_n|)q(X_n), \quad m \neq n, \\
    \alpha_0/h^2q(X_n) \quad, m = n, 
    \end{cases} 
\end{equation*}
and $\alpha_0$ is a diagonal correction term and is computed as in Lemma 3.9 in \cite{duan2009high}.   

With $(Kq)^d$ at hand, it is straightforward to solve the forward scattering problem using the Lippmann-Schwinger formulation. For a discretized incoming wave $\eta_i^d$ one solves the discretized equation 
\begin{equation}
    (I + k_i^2(Kq)^d)\eta^d = \eta^d_i.
    \label{eq: discrete LS}
\end{equation}
The waves plotted in Figure \ref{fig: 3dwave} and \ref{fig: depthandwaves} are found by solving the above problem, and due to the high-order method and low wavenumbers involved, the discretization parameter $N$ does not have to be very large to obtain accurate results.

\subsection{Experiments}
For the numerical experiments, we consider two different topographies, and for each, we consider incoming waves with two different frequencies.  The first, shallower topography $d_1$ is shown in Figure \ref{fig: depthandwaves}, and the second, deeper topography $d_2$ is shown in Figure \ref{fig: depthandwaves2}, both together with the resulting scattered wave with $\omega = 1$, and both satisfy the mild-slope condition \eqref{eq: mild slope condition}. The following summarizes the setup of the numerical experiments:
\begin{itemize}
    \item The measurement domain $\mathcal{D}$ is a square with side lengths $L = 100 m$.
    \item For each topography $d_1,d_2$ we consider incoming waves with angular velocities $\omega_1 = 1$ and $\omega_2 = 2$. 
    \item The measurement grid $\mathcal{D}^d$  is a set of $100\times 100$ equispaced points in $\mathcal{D}$ as in \ref{Section: discrete inversion}. The measurement $\mathcal{M}^d$ is found by linearly interpolating the solution $\eta$ computed on a finer grid and sampling it on $\mathcal{D}^d$. 
    Each depth profile is on the form $d_j = H_{0,j} - h_j$, and we assume the wavenumbers of the incoming waves are one the form $\bm{k} = (k,0)$.  We therefore have for each depth profile the incoming wave $\eta_{i,j,n} = a\exp(\bm{k}_{i,j,n} \cdot X)$  
    where 
    $$ k_{i,j,n} \text{ satisfies }  \frac{\omega_n^2}{g} = k_{i,j,n}\tanh(k_{i,j,n} H_{0,j}), \quad j=1,2, \quad n = 1,2.$$
    As $H_{0,1} = 1.5 $ and $H_{0,1} = 2.5$, we get 
    $$ k_{i,1,1} = 0.267, \quad k_{i,1,2} = 0.581, \quad , k_{i,2,1} = 0.211 ,\quad k_{i,2,2} = 0.486 .$$
    As we are in the linear regime, the amplitude $a$ should fulfill $ak \ll 1$, and we can take $a=0.3$.  
   
    \item In reality, the water surface will, under most conditions, deviate quite a lot from the smooth solution $\eta$ of the mild-slope equation.  Small, local ripples generated by wind, turbulence, non-linear effects like wave-breaking, optical reflections, and the fact that the full wavefield is likely to be a superposition of many waves make abundant the possible sources for the measurement $\mathcal{M}^d_\varepsilon$ to deviate from the model wave $\eta$. It is beyond the scope of this paper to investigate these matters properly. Instead, we choose the simplest strategy, and add $10 \% $ relative noise to the measurements. The noise will be taken from a  identically distributed, independent Gaussian distribution, i.e.,   $\varepsilon^d \sim \mathcal{N}(0,I_{d}\sigma^2)$, and scaled such that 
    $$ \frac{\|\varepsilon^d\|_2}{\|\mathcal{M}^d\|_2} = 0.1 \quad \text{and} \quad \mathcal{M}_\varepsilon^d = \mathcal{M}^d +\varepsilon^d .$$
    \item We solve the inverse problem with data $\mathcal{M}_\varepsilon^d$ as outlined in Section \ref{Section: discrete inversion}, and denote the reconstructed wavenumbers corresponding to the incoming wave $k_{i,j,n}$ by $k_{\varepsilon,j,n}$ (we the same labeling for other reconstructed quantities). For the reconstruction from incoming waves with $\omega_1$ we take $\sigma = 2.5 $ and $\gamma = 0.001$, while for $\omega_2$ with take $\sigma = 1.5$ and $\gamma = 0.001$. In both cases we use $\alpha = 0.1$. 
    \item Results are shown Figures \ref{fig: rec1}, \ref{fig: rel_err} and \ref{fig: abs_err}, and error norms are given in Table \ref{table: errors}. All reconstructions have been truncated near the edges, due to the numerically induced errors there. 
\end{itemize}
\begin{figure}[H]
    \centering
    \includegraphics[width = 1\textwidth]{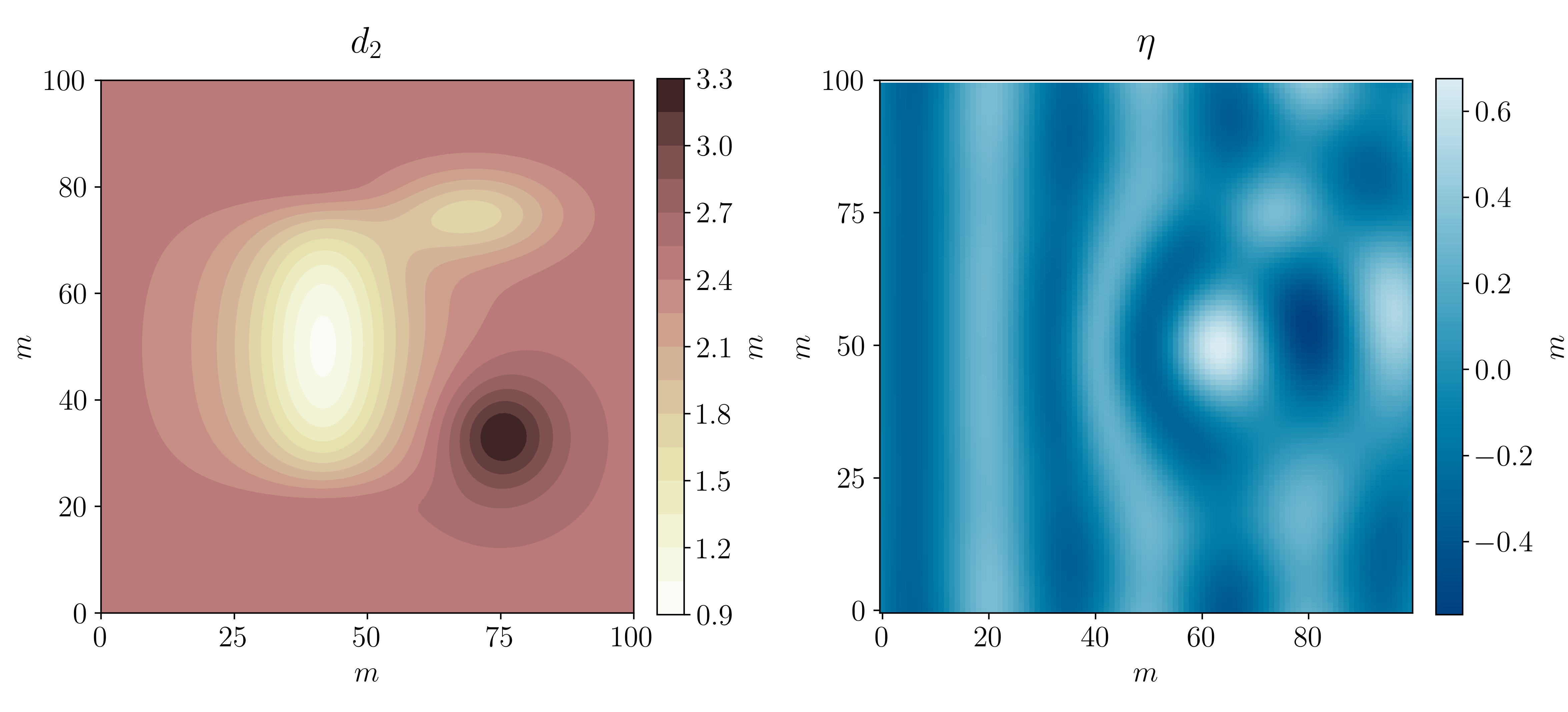}
    \caption{To the left is a contour plot of the deeper topography $d_2(X)$, and to the right, we see the water wave $\text{Re}\{\eta(X)\}$. }
    \label{fig: depthandwaves2}
\end{figure}

\subsubsection*{Results from numerical experiments}
Figure \ref{fig: rec1} shows the true topographies together with the reconstructed topographies for the different frequencies of incoming waves. In Figure \ref{fig: rel_err}, we plot the relative error for the same reconstructed topographies, and Figure \ref{fig: abs_err} shows the absolute error. 

For the longer incoming waves ($\omega = 1$), we see that we get satisfactory reconstructions for both the shallow and deep topographies, and that the whole depth range is reconstructed quite well. In the reconstruction of $d_2$, we see that both the relative and absolute error increase slightly at lower depths. This seems to be a consequence of the fact that lower depths correspond to higher wavenumbers, and that the low-pass filter we use suppresses the higher frequencies more. 
Moreover, one can see certain artifacts in the reconstructions corresponding to $\mathcal{M}_\varepsilon \approx 0$. These artifacts appear mostly near the edges, but also some places in the interior, for example in the deepest region of $d_2$. In potential applications, such artifacts can easily be identified and removed. 

In the reconstructions from the shorter waves $(\omega = 2)$, results are not so good. For the shallow waters, the topography is still visible, but the relative error is close to $100\%$ in the shallowest regions. For the deep topography, we see that we basically get no useful information, indicating that the combination of wavenumber and depth is outside the range of applicability of the method. The fact that the error is small in the deeper region results from the fact that the $d_\alpha(k_\varepsilon) \approx H_0$ when $k_\varepsilon < \mu + \alpha$. Had this not been the case, the error could have been much larger, but the assumption that the bottom satisfies the mild-slope condition limits the maximal depth difference in a given region, and so it also limits the error in $d_\alpha$. 
In Figure \ref{fig: shortwave}, the waves corresponding to $\omega = 2$ are plotted. When comparing to the very visible scattering patterns observed in Figures \ref{fig: depthandwaves} and \ref{fig: depthandwaves2}, it is clear that the influence from the bottom is less visible in the shorter waves, giving an intuitive explanation for the poor reconstruction results from the shorter waves.  

We have also conducted the numerical experiments using a numerical method based on the integral equation formulation in \eqref{eq: bandK} and using Tikhonov regularization, and the results were quite similar to those presented above but much more cumbersome to implement, as it requires generating huge linear systems. 

\begin{table}
\centering
\begin{tabular}{c c c c} 
 $\dfrac{\|d_1 - d_\alpha(k_{\varepsilon,1,1})\|_2}{\|\mathcal{M}_{1,1}\|_2}$ & $\dfrac{\|d_2 - d_\alpha(k_{\varepsilon,1,2})\|_2}{\|\mathcal{M}_{1,2}\|_2}$ & $\dfrac{\|d_1 - d_\alpha(k_{\varepsilon,2,1})\|_2}{\|\mathcal{M}_{2,2}\|_2}$ & $\dfrac{\|d_2 - d_\alpha(k_{\varepsilon,2,2})\|_2}{\|\mathcal{M}_{2,2}\|_2}$ \\  
 \hline
$0.27$ & $0.58$ & $0.60$ & $1.3$ \\ 
\end{tabular}
\caption{}
\label{table: errors}
\centering
\end{table}

\begin{figure}
    \centering
     \makebox[\textwidth][c]{\includegraphics[width = 1.05\textwidth]{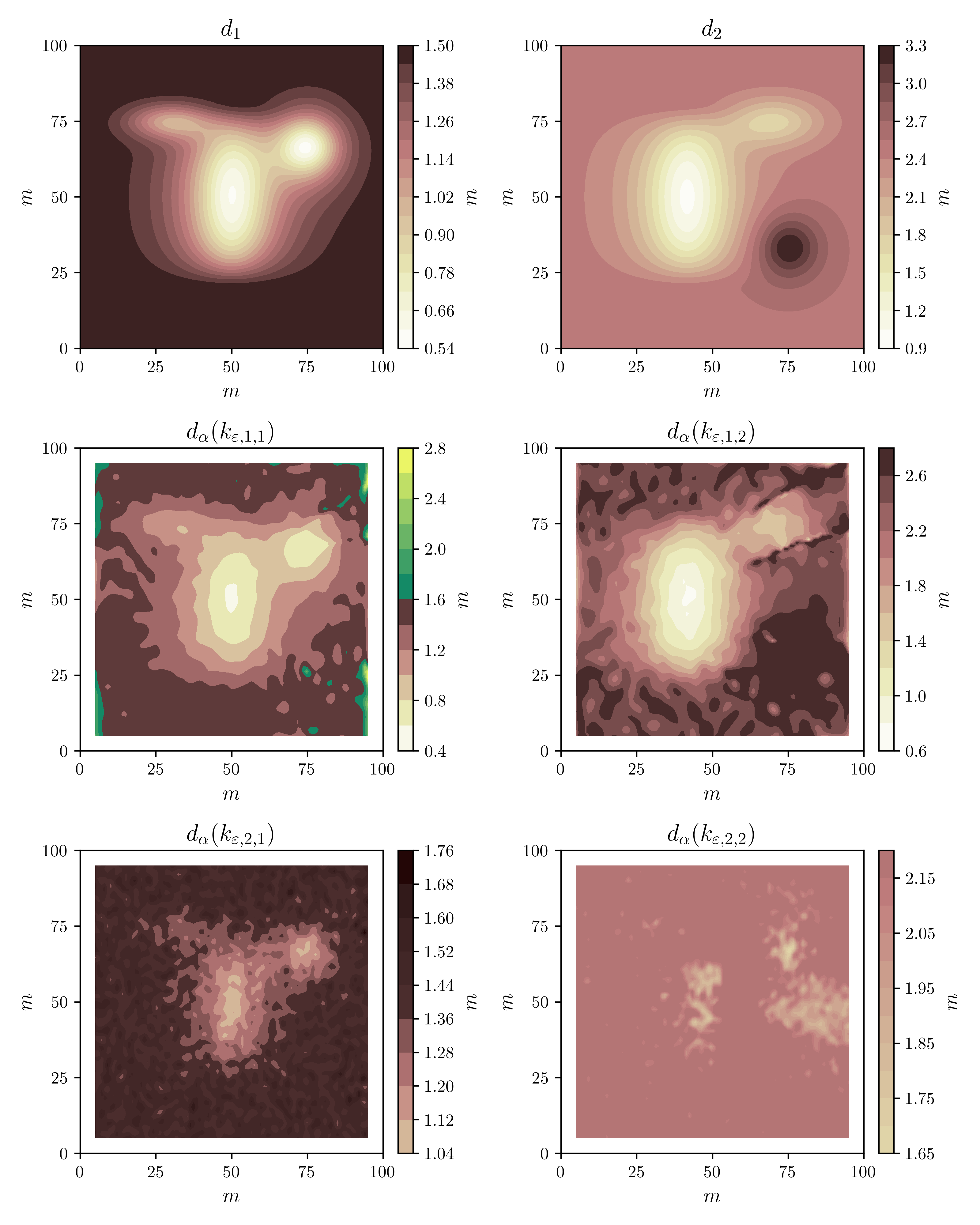}}
    \caption{On the top left and right are contour plots of the depth functions $d_j(X)$ on a $100 \times 100$ meter domain. In the middle row are the reconstructions of the profiles $d_\alpha(k_\varepsilon)$ from the incoming waves with $\omega = 1$, and in the bottom row, the reconstructions $d_\alpha(k_\varepsilon)$ from the incoming waves with $\omega = 2$. } 
    \label{fig: rec1}
\end{figure}

\begin{figure}
    \centering
     \makebox[\textwidth][c]{\includegraphics[width = 1.2\textwidth]{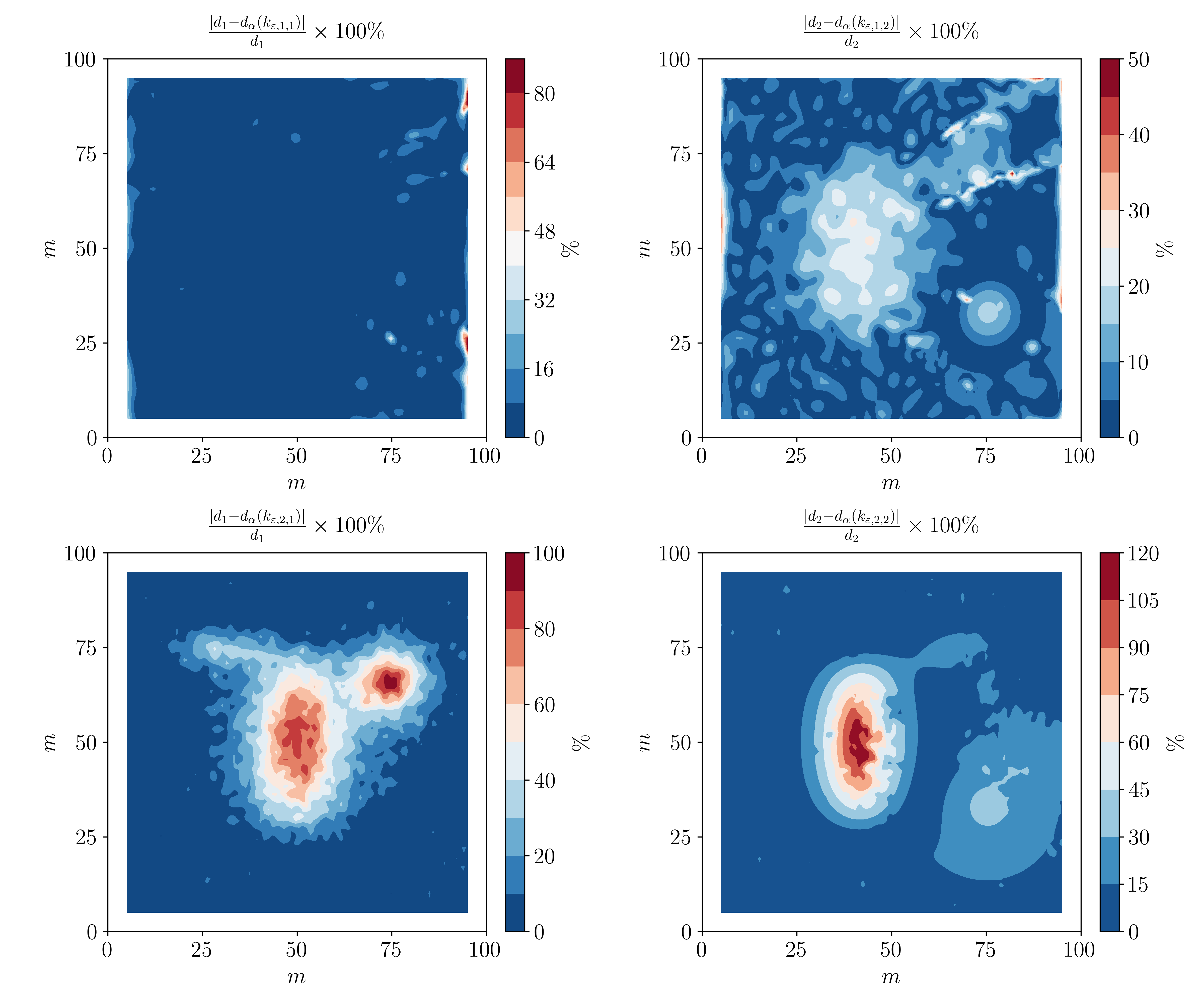}}
    \caption{Relative errors (in percentages) of the reconstructions shown in Figure \ref{fig: rec1}.}
    \label{fig: rel_err}
\end{figure}

\begin{figure}
    \centering
     \makebox[\textwidth][c]{\includegraphics[width = 1.2\textwidth]{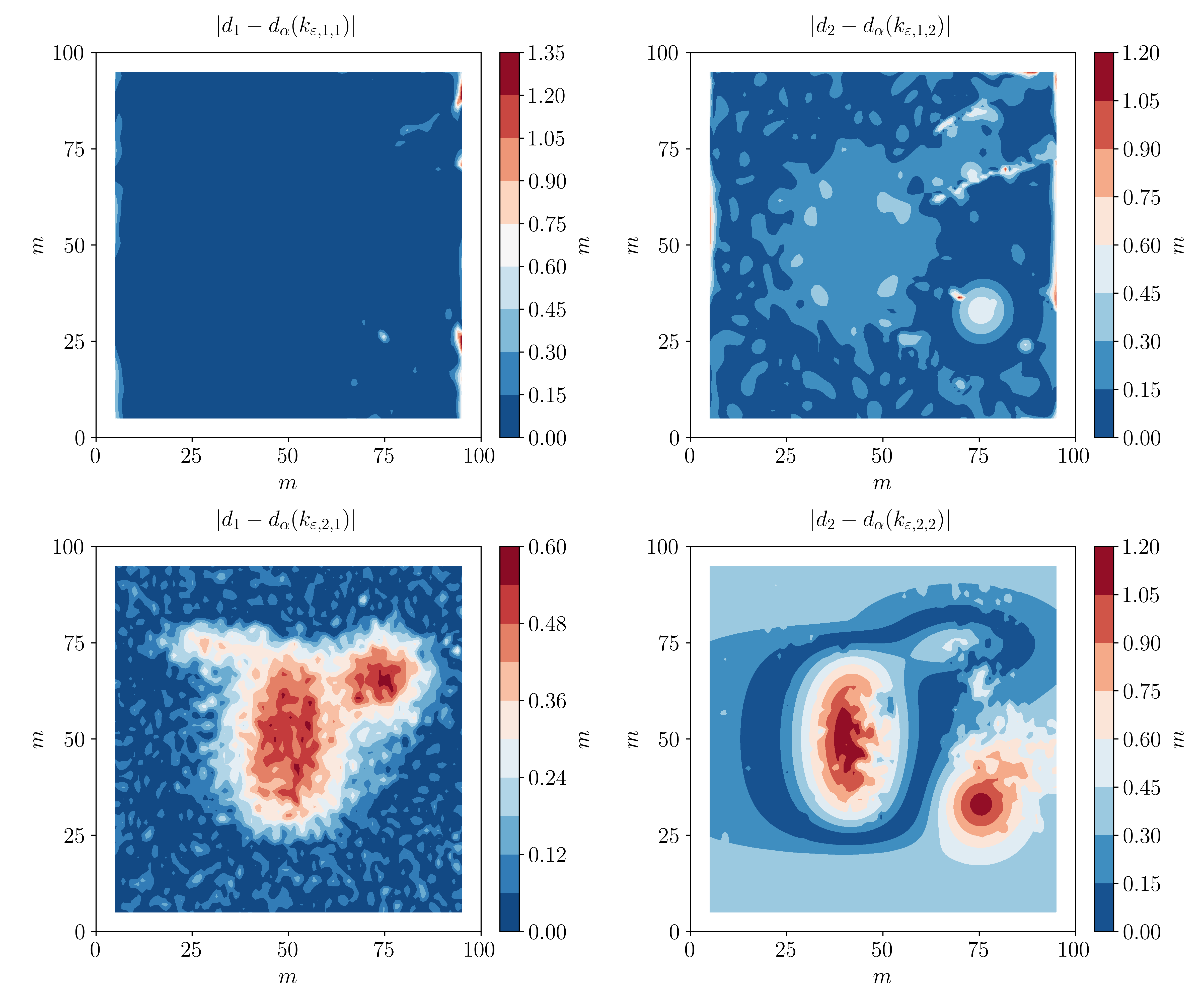}}
    \caption{Absolute errors of the reconstructions shown in Figure \ref{fig: rec1}. }
    \label{fig: abs_err}
\end{figure}

\begin{figure}
    \centering
     \makebox[\textwidth][c]{\includegraphics[width = 1.2\textwidth]{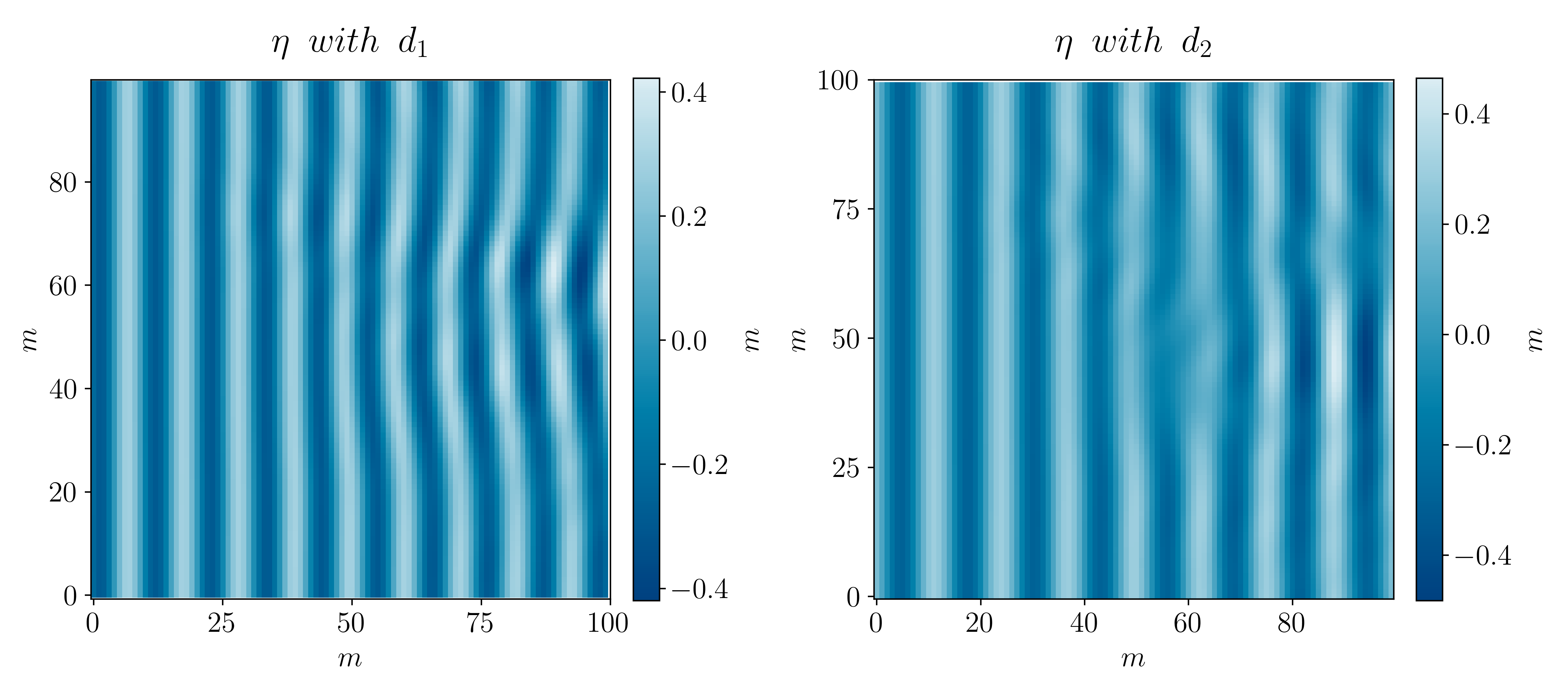}}
    \caption{The plot shows the real part of the wave $\eta$ corresponding to the frequency $\omega = 2$. On the left is the wave propagating over the topography $d_1$, and on the right the wave is propagating over the topography $d_2$. }
    \label{fig: shortwave}
\end{figure}

\section{Conclusion}
\label{sect: conclusion}

We have analyzed the problem of reconstructing the seabed topography from observations of water waves scattered by the depth variations. We modeled the waves using the mild-slope equation and showed that for longer waves, the equation can be simplified, making it more suitable for inversion. We then showed that the inverse problem is conditionally Hölder-stable, with a constant that depends explicitly on the critical parameters $\alpha, \gamma$, and $\mu$. Moreover, we proposed a simple and fast inversion method and tested it on synthetic data, and the results confirmed the analysis in Sections \ref{sect: wave scattering}-\ref{sect: inversion} and showed that the inversion method works quite well. 

We may therefore conclude that for physical conditions satisfying the mild-slope equation, and if the incoming waves are sufficiently long compared to the depth, it is possible to stably reconstruct the seabed topography.

We hope that the approach taken to estimating seabed topography in this paper will be of interest both to oceanographers and mathematicians (and others). For real applications, it gives a simple reconstruction method together with a concrete set of assumptions on the waves and seabed for the method to be applicable and give reliable results. For mathematicians, the analysis of inverse problems for water waves presents many interesting problems; specifically, the large number of different asymptotic models that exist for water waves, each describing different regimes, should be of interest. The fact that full wave field can be observed means that one has access to more information than for the usual inverse problems for, e.g., electromagnetic or acoustic waves, where one usually measures the wave field at points or on a surface. In addition, there are many other physical processes in the ocean that influence the surface waves. Currents, internal waves, and turbulence are examples of processes that are difficult to measure directly but might be possible to estimate due to their effect on the surface waves, giving rise to interesting and challenging inverse problems.

\section*{Appendix}

This section contains the proofs of Propositions \ref{prop: uniqueness}-\ref{prop: tikhonov} and the derivation of the estimates used in the proof of Proposition \ref{prop: simplification}. 

\begin{proof}[\textbf{Proof of Proposition \ref{prop: K}:}]
The injectivity of $K$ follows from the uniqueness of the source problem for the Helmholtz equation: Assume there are $v \neq \tilde{v}$ both given by $K\tilde{u} $. Then $ v - \tilde{v} = V|_\mathcal{D}$, where $V$ satisfies  $(\Delta + k^2_i)V = 0$ in $\R^2$ and the radiation condition, and consequently $V = 0$. 
The compactness follows from the fact that $K: L^2(\R^2) \to H^2_{loc}(\R^2)$ and that $H^2(\mathcal{D})$ is compactly embedded in $L^2(\mathcal{D})$. 

To estimate the singular values of $K$, we rely on Theorem 2.6 in \cite{koch2020instability}. It states that for a (classical) compact pseudo-differential operator (PDO) ${A: L^2(M) \to L^2(M)}$ of order $-m$, where $M$ is a smooth $n$-dimensional manifold, and such that the principal symbol is non-vanishing for some $(x_0,\xi_0) \in T^*M$, the singular values $s_k(A)$ of $A$ satisfies ${ c k^{-m/n} \leq  s_k(A) \leq Ck^{-m/n}}$, for constants $0<c\leq C$. (Both the necessary background on PDOs and the results used in the rest of the proof are found in chapters 6 \& 7 in \cite{hsiao2008boundary})

To apply this theorem, we write $K = \chi_{\mathcal{D}}V$, where $Vu(X) = \int_{\mathcal{D}} G_{k_i}(X,Y) u(Y) \mathrm{d}Y$ and $\chi_{\mathcal{D}}$ is the characteristic function for $\mathcal{D}$. Following the development in Chapter 7, \cite{hsiao2008boundary}, we note that the kernel of $V$, $G_{k_i}(X,Y) = \frac{i}{4}H^{(1)}_0(k|X-Y|) = g(z) $, is pseudohomogeneous function of order $0$ with respect to  $z = X-Y$. Consequently, we can use Theorem 7.1.1 in \cite{hsiao2008boundary} that identifies when an integral operator is a PDO, and since we are in $\R^2$, we conclude that the operator $V$ is a classical PDO of order $-2$.  

Next, let $\mathcal{D}_\delta \subset \mathcal{D}$ be a smooth domain such that $0 < \text{dist}(\partial \mathcal{D}_\delta, \partial \mathcal{D}) < \delta $, and let $\varphi_\delta \in C^\infty_0(\mathcal{D})$ be a characteristic function on $\mathcal{D}_\delta$. We define ${K_\delta = \varphi_\delta(x)V}$. Then $K_\delta : L^2(\mathcal{D}) \to L^2(\mathcal{D})$ is a PDO of order $-2$, and its symbol is ${q \sim \varphi_\delta(x)(k_i^2 - |\xi|^2)^{-1}g(\xi)}$, where $g$ is a smooth cut-off function such that $g(\xi) = 0 $ for $|\xi| = k_i$ and $g(\xi) = 1$ for $|\xi| \gg k_i$. Hence, the principal symbol is non-vanishing for some $(x_0,\xi_0) \in T^*\mathcal{D}$ and we can apply the result from \cite{koch2020instability} to $K_\delta$. We now use the fact that we can bound the difference between the singular values of compact operators by the difference of the operators (Corollary 2.3, \cite{gohberg1978introduction}), i.e., 
\begin{equation}
    |s_n(K) - s_n(K_\delta) | \leq \|K - K_\delta\|_{\mathcal{L}(L^2(\mathcal{D}),L^2(\mathcal{D}))}, \quad n = 1,2,3... 
    \label{eq: singval bound}
\end{equation}   
We write $v_f = Vf$. Moreover, 
\begin{align*}
    \|v_f\|_{L^\infty(\mathcal{D})} &=  \sup_{X \in \mathcal{D}}\bigg| \int_{\mathcal{D}_\delta }G_{k_i}(X,Y) f(y) \mathrm{d}Y \bigg| \\
    & \leq \sup_{X \in \mathcal{D}} \|G_{k_i}(X,\cdot)\|_{L^2(\mathcal{D})}\|f\|_{L^2(\mathcal{D})}.
\end{align*} 
As $X \mapsto \|G_{k_i}(X,\cdot)\|_{L^2(\mathcal{D})}$ is continuous it has a maximum $X_m$ in $\overline{\mathcal{D}}$. Let $B_{X_m,R}$ be a disk of radius $R$ centered at $X_m$ and containing $\overline{\mathcal{D}}$. As $|H_0^{(1)}(z)|^2 = J_0^2(z) + Y_0^2(z)$, we have for $R$ large enough that  
\begin{align*} \|G_{k_i}(X_m,\cdot)\|_{L^2(\mathcal{D})}^2  &\leq \frac{1}{16}\int_0^R|H^{(1)}_0(k_ir)|^2 r \mathrm{d}r \\
&= \frac{R^2}{32}\left( J_0^2(k_iR) + J_1^2(k_iR) + Y_0^2(k_iR) + Y_1^2(k_iR) \right)  \leq \frac{R}{\pi k_i}.
\end{align*}
Here, $J_n,Y_n$ are, respectively, Bessel functions of the first and second kind and order $n$, and the last inequality follows from their asymptotic expansions $J_n^2(r) \sim Y_n^2(r) \leq 2\times \frac{2}{\pi r}$ for large $r $. (cf. Chapter 9 in \cite{abramowitz1968handbook}). 
We then compute
\begin{align*}
    \|K - K_\delta\|_{\mathcal{L}(L^2(\mathcal{D}),L^2(\mathcal{D}))} &= \sup_{\|f\|_{L^2(\mathcal{D})} \leq 1} \| v_f(1 - \varphi_\delta) \|_{L^2(\mathcal{D})} \\
    &\leq \sup_{\|f\|_{L^2(\mathcal{D})} \leq 1}  \|v_f\|_{L^\infty(\mathcal{D})}\|1 - \varphi_\delta \|_{L^2(\mathcal{D})} \\
     &\leq \sqrt{\frac{R}{\pi k_i}}\|1 - \varphi_\delta \|_{L^2(\mathcal{D})}. 
\end{align*} 
As the term  $\|1 - \varphi_\delta \|_{L^2(\mathcal{D})}$ can be made arbitrary small by decreasing $\delta$, it follows from \eqref{eq: singval bound} that 
$$ |s_n(K) - s_n(K_\delta)| \leq \varepsilon \quad n = 1,2,3,... $$
for any $\varepsilon > 0$, and so there exist $0 < c \leq C$ such that $$cn^{-1} \leq s_n(K) \leq Cn^{-1} \quad \text{for } n = 1,2,3,... \hspace{1mm}.$$ 
\end{proof}

\begin{proof}[\textbf{Proof of Proposition \ref{prop: uniqueness}:}]
    As $K$ is injective, there exists a unique solution $u = \mathcal{M}q$ to the equation $k_i^2Ku = b = \mathcal{M} - \eta_i$. Consequently $q = u/\mathcal{M}$ and $k(X) = k_i\sqrt{1- q(X)} $, and so $d(X) = \frac{\tanh^{-1}(\mu/k(X))}{k(X)}$. 

    For the local uniqueness result, we apply the unique continuation property. The unique continuation property states that if $u \in H^1(\tilde{\mathcal{D}})$ satisfies $Lu = 0$, where $\tilde{\mathcal{D}} \in \R^n$ is a bounded Lipschitz domain and $L$ is a second-order elliptic differential operator with continuous coefficients, then $u|_\mathcal{O} = 0$ implies $u = 0$, where $\mathcal{O} \in \mathcal{D}$ is any non-empty, open set (cf. Section 2 in \cite{choulli2016applications}).
    
    As the wave $\eta$ satisfies the MSE, and the MSE is a second-order elliptic PDE with continuous coefficients, we may apply unique continuation property to conclude that for any $\mathcal{O}$ is any non-empty open set in $\R^2$, a local measurement $\mathcal{M} = \eta|_{\mathcal{O}}$ uniquely determines $\eta$ on the set $\mathcal{D}$, and so it uniquely determines $d$.  
\end{proof}

\begin{proof}[\textbf{Proof of Proposition \ref{prop: tikhonov}:}]
We write \eqref{eq: fredholm} as $Ku = b/k_i^2$ and note that $$\|b_\varepsilon/k_i^2 - b/k_i^2\|_{L^2(\mathcal{D})} = \|k_i^{-2}(\mathcal{M}_\varepsilon - \mathcal{M}) \|_{L^2(\mathcal{D})} \leq \delta/k_i^2.$$ Let $u_\delta^\lambda$ be the solution to \eqref{eq: tikhonov}. The approximation error $\|u_\delta^\lambda - u\|_{L^2(\mathcal{D})}$, between the regularized and true solution to $Ku = b$ is bounded when the true solution $u$ satisfies a so-called source condition. For example, Theorem 8.4 in \cite{hanke2017taste} states that if $R(K)$ is dense in $L^2(\mathcal{D})$, $\|b_\varepsilon - b \|_{L^2(\mathcal{D})} \leq \delta $ for some $\delta > 0$ and $u$ is on the form $u = K^*w$, where $K^*$ is the adjoint of $K$ and $w \in L^2(\mathcal{D}))$, then there exists $\lambda > 0$ such that $\| u_\delta^\lambda - u \|_{L^2(\mathcal{D}} \leq 2 \sqrt{\|w\|_{L^2(\mathcal{D})} }\sqrt{\delta}.$

We now show that $u$ satisfies this source condition. It is a quick computation to show that the adjoint of $K$ is $K^*w = \int_\mathcal{D} \overline{G}(X,Y) w(Y) \mathrm{d}Y$, where $\overline{G}(X,Y)$ denotes the complex conjugate. In fact,  $\overline{G}(X,Y) = -\frac{i}{4}H^{(2)}_0(k_i|X-Y|)$ is the Hankel function of second kind and order $0$. This is the fundamental solution to the Helmholtz equation with \emph{incoming} radiation conditions \cite{hubert2012vibration, schot1992eighty}. Therefore $v = K^*w$ is given by  $ v = V|_\mathcal{D}$ where 
$V$ is the unique solution to $(\Delta + k_i^2)V = -w $ in $\R^2$ satisfying the incoming radiation condition $$\lim_{|X| \to \infty } \sqrt{|X|}\left( \partial_{|X|} + ik\right) V = 0, \quad \text{uniformly for } X/|X| \in S^1.$$ 
We show that there exists a $w \in L^2(\mathcal{D})$ such that $u = \eta q = K^*w$. 
For the contrast function $q = 1 - \frac{k^2}{k_i^2}$ we have $\text{supp}(q) = \text{supp}(h) = \Omega \subset \mathcal{D}$ and so $\text{supp}(u = \eta q) \subset \mathcal{D}.$ By assumption of Proposition \ref{prop: wellposed}, $d(X) \in C^2(\Omega)$, and so is $q$, and in the proof of Proposition \ref{prop: simplification}, we saw that $\eta(X) \in C^1(\R^2)$. Consequently, $u = \eta q \in C^1_0(\Omega)$. As in the proof of Proposition \ref{prop: simplification}, $\partial_{x_i}G_{k_i}(X,Y) \in L^2_{loc}(\R^2)$, 
and so differentiation of the convolution operator satisfies $$\partial_{x_i} \int_{\R^2}G_{k_i}(X,Y) u(Y) \mathrm{d}Y =  \int_{\R^2}\partial_{x_i}G_{k_i}(X,Y) u(Y) \mathrm{d}Y = \int_{\R^2}G_{k_i}(X,Y) \partial_{x_i}u(Y) \mathrm{d}Y.$$
It follows that
$$ \partial^2_{x_i,x_j} \int_{\R^2}G_{k_i}(X,Y) u(Y) \mathrm{d}Y = \int_{\R^2} \partial_{x_j}G_{k_i}(X,Y) \partial_{x_i}u(Y) \mathrm{d}Y, $$
the since $\partial_{x_j}G_{k_i}(X,Y)$ is weakly singular on $\R^2$, the above operator maps $C(\R^2)$ to $C(\R^2)$. Recalling that 
$$\eta = \eta_i - k_i^2\int_{\R^2}G_{k_i}(X,Y) u(Y) \mathrm{d}Y,$$
it follows that $\partial^2_{x_i,x_j} \eta \in C(\R^2).$ Therefore, $u = \eta q \in C^2_0(\Omega)$, and so $u$ is the solution to $(\Delta +k_i^2)u = -\tilde{w}$, where $\tilde{w} = -(\Delta + k_i^2)u \in L^2(\Omega)$. Letting $w$ be the extension by zero of $\tilde{w}$ to $\mathcal{D}$, we conclude that there exist $w \in L^2(\mathcal{D})$ such that  $ u = K^*w$. It also follows that $\|w\|_{L^2(\mathcal{D})} \leq C \|\eta q\|_{H^2(\mathcal{D})}$.  
Last, as $\overline{R(K)} = N(K)^\perp = L^2(\mathcal{D})$, we may apply Theorem 8.4 in \cite{hanke2017taste} and conclude that there exists some regularization parameter $\lambda$ such that our proposition holds. 
\end{proof}

Before we give the proof of Proposition \ref{prop: stability}, we need the following result. 

\begin{proposition}
    Let $\gamma > 0$ be such that $\mathcal{D}_\gamma $ is non-empty. Then there is some constant $C> 0$ depending on $q,\varepsilon$ and $\gamma$ such that 
    $$ \|q - q_\varepsilon \|_{L^1(\mathcal{D}_\gamma)} \leq C \|\mathcal{M} - \mathcal{M}_\varepsilon\|_{L^2(\mathcal{D})}. $$
    \label{prop: q_stability}
\end{proposition}
\begin{proof}
   
As $|\mathcal{M}_\varepsilon| \geq \gamma$, we have 
\begin{equation*}
        q_\varepsilon = \frac{u_\delta^\lambda}{\mathcal{M}_\varepsilon} = \frac{\mathcal{M}q + \tilde{\varepsilon}}{\mathcal{M} + \varepsilon} = \frac{\tilde{\varepsilon}}{\mathcal{M} + \varepsilon} + \frac{\mathcal{M}q}{\mathcal{M} + \varepsilon}.
\end{equation*}
Hence 
$$ q_\varepsilon - q =   q\left( \frac{\mathcal{M}}{\mathcal{M} + \varepsilon}  - 1\right) + \frac{\tilde{\varepsilon}}{\mathcal{M} + \varepsilon} = q \frac{-\varepsilon}{\mathcal{M} + \varepsilon}   + \frac{\tilde{\varepsilon}}{\mathcal{M} + \varepsilon}$$
and so
$$ |q_\varepsilon - q| \leq   \|q\|_\infty \frac{|\varepsilon|}{\gamma} + \frac{|\tilde{\varepsilon}|}{\gamma}.$$
Integrating, we find 
\begin{align*}
    \int_{\mathcal{D}_\gamma}|q_\varepsilon - q| \mathrm{d}X &\leq |\mathcal{D}|^{1/2}\gamma^{-1}\left(\|q\|_\infty \|\varepsilon\|_{L^2(\mathcal{D})} + \|\tilde{\varepsilon}\|_{L^2(\mathcal{D})} \right) \\
    &\leq |\mathcal{D}|^{1/2}\gamma^{-1}\left( \|q\|_\infty + C \sqrt{\frac{\|u\|_{H^2(\mathcal{D})}}{k_i^2} } \right)\|\varepsilon\|_{L^2(\mathcal{D})}, 
\end{align*}
where we used the result from Proposition \ref{prop: tikhonov}. Above $|\mathcal{D}|$ denotes the area of $\mathcal{D}$.
\end{proof}

We may now prove Proposition \ref{prop: stability}. 

\begin{proof}[\textbf{Proof of Proposition \ref{prop: stability}:}]
Let $k_\varepsilon = \sqrt{|1-q_\varepsilon|}k_i \geq 0$. As $k  \geq 0$, we have 
$$(k_\varepsilon - k)^2 \leq |k_\varepsilon^2 - k^2| = k_i^2|q_\varepsilon - q|.$$
Recalling that $d_\alpha$ (Definition \ref{def: dalpha}) has Lipschitz constant 
$M_{\mu + \alpha, \mu}$, we get
\begin{align*}
\int_{\mathcal{D}_\gamma} |d_\alpha(k_\varepsilon) -d_\alpha(k)|^2 \mathrm{d}X &\leq \ M_{\mu + \alpha, \mu}^2\int_{\mathcal{D}_\gamma} | k_\varepsilon - k|^2 \mathrm{d}X  \\
&=  k_i^2 M_{\mu + \alpha, \mu}^2\int_{\mathcal{D}_\gamma} | q_\varepsilon - q| \mathrm{d}X.
\end{align*} 
Writing $d_\alpha^\varepsilon = d_\alpha(k_\varepsilon)$ and using Proposition \ref{prop: q_stability} and the upper bound from \eqref{eq: lipschitz}, we find
\begin{equation*}
    \|d_\alpha^\varepsilon - d_\alpha  \|_{L^2(\mathcal{D}_\gamma)} \leq C\gamma^{-1/2}\left(\frac{3k_i}{\alpha(\alpha + \mu)}\right) \sqrt{\|\varepsilon\|_{L^2(\mathcal{D})}},
\end{equation*}
where the constant $C$ satisfies 
\begin{equation*}
C =  \left(|\mathcal{D}|^{1/2}\left( \|q\|_\infty + \tilde{C} \sqrt{\|q\eta\|_{H^2(\mathcal{D})}/k_i^2} \right)\right)^{1/2}.
\end{equation*}
\end{proof}

\begin{proof}[\textbf{Derivation of estimates in Proposition \ref{prop: simplification}}]

We first estimate the term $\frac{\nabla c}{c}$. Computing the gradient of the dispersion relation \eqref{eq: dispersion}, we get 
\begin{equation}
    \frac{\nabla k}{k} = - p(kd)\frac{\nabla d}{d}, \quad  \text{with} \quad p(t)\frac{t\text{sech}^2(t)}{\tanh(t) + t\text{sech}^2(t)}.
    \label{eq: gradient ratios}
\end{equation}
Next, we find that 
\begin{equation}
    \frac{\Delta k}{k} = (p(kd) - p^2(kd))\frac{|\nabla d|^2}{d^2} +p'(kd)kd(1-p(kd))\frac{|\nabla d|^2}{d^2} + p(kd)\frac{\Delta d}{d}.
    \label{eq: laplace ratios}
\end{equation}
Disregarding the constant term $\omega^2/c_0$, we write $c = \frac{g(kd)}{k^2}$, where $g(x) = \frac{1}{2}\left(1 +\frac{2x}{\sinh(2x)}\right).$  Then 
$$ \nabla c =  \frac{-2g(kd)\nabla k  +k g'(kd)(\nabla kd + k\nabla d)}{k^3}  $$ 
and so 
\begin{align*}
    \frac{\nabla c}{c} & = \frac{2g(kd)\nabla k +k^2dg'(kd)(\nabla k/k + \nabla d/d)}{g(kd)k} \\
                       &= \frac{\nabla k}{k} + \frac{kd g'(kd)}{g(kd)}(\nabla k/k + \nabla d/d). \\
\end{align*}
Consider now the function $s(t) = t\frac{g'(t)}{g(t)}$. It is bounded and ${s(0) = \lim_{x \to \infty} s(x) = 0}$, and numerical evaluation gives $ \max |s(t)| \leq 1/2$. Similarly, $|p(t)| \leq 1/2$, and by \eqref{eq: gradient ratios} we get that 
\begin{equation*}
    \frac{|\nabla c|}{c} \leq  \frac{|\nabla d|}{2d} + s(kd)\frac{3|\nabla d|}{2d}  \leq \frac{5|\nabla d|}{4d}.
\end{equation*}
Next, we note that    
$$\frac{\Delta c^{1/2} }{c^{1/2}} = - \frac{|\nabla c|^2}{4c^2} + \frac{\Delta c}{2c}.$$
For $\Delta c$, we have 
\begin{equation*}
    \Delta c = \frac{k^2 \Delta g }{k^4} - \frac{g \Delta k^2}{k^4} - \frac{2k^2\nabla c \cdot \nabla k^2}{k^4}.
\end{equation*}
For the last term we get 
$$  \frac{2k^2\nabla c \cdot \nabla k^2}{k^4c} = \frac{4\nabla c \cdot \nabla k}{kc} = \frac{4\nabla c}{c} \cdot \frac{\nabla k}{k} \leq \frac{5|\nabla d|^2}{2d^2}.  $$
For the second term we find that
\begin{equation*}
    \frac{g \Delta k^2 }{k^4c} = \frac{2 |\nabla k|^2 + k \Delta k}{k^2} \leq \frac{|\nabla d|^2}{d^2} + \frac{|\Delta k|}{k},
\end{equation*}
and for the first term
\begin{equation*}
    \frac{k^2\Delta g}{k^4c} = \frac{\Delta g }{g} = \frac{g''}{g}k^2d^2 |\nabla k/k + \nabla d/d|^2 + \frac{g'}{g}kd\left(\Delta k/k + 2\nabla k \cdot \nabla d /kd + \Delta d/d \right).  
\end{equation*}
Again, we have the coefficient $s(kd)$ in front of the last term and with $r(t) =\frac{g''(t)}{g(t)}t^2 $ we have $r(kd)$ in front of the first. Similar analysis shows that $|r(t)| < 11/10 $, and so
\begin{equation*}
    \frac{k^2\Delta g}{k^4c} \leq  3\frac{|\nabla d|^2}{d^2} + \frac{1}{2}(\Delta k/k + \Delta d/d).
\end{equation*}
From the analogous analysis of equation \eqref{eq: laplace ratios}, we get 
\begin{equation*}
    \frac{|\Delta k|}{k} \leq \frac{1}{2}\left(\frac{|\nabla d|}{d^2} + \frac{|\Delta d|}{d}\right).  
\end{equation*}
Summing it all up, we find that 
\begin{equation*}
   \frac{|\Delta c^{1/2}|}{ c^{1/2}} \leq 4\frac{|\nabla d|^2}{d^2} + \frac{3}{4} \frac{|\Delta d|}{d}.
\end{equation*}
We may now estimate $\|q\|_\infty$:
\begin{equation*}
    k_i^2|q(X)| = k_i^2 + |\tilde{k}^2| = k_i^2 + k^2 + \frac{|\Delta c^{1/2}|}{ c^{1/2}} \leq k_i^2 + k^2 + 4\frac{|\nabla d|^2}{d^2} + \frac{3}{4} \frac{|\Delta d|}{d}.
\end{equation*}
By invoking mild-slope assumption we have $|\nabla d|^2/d^2 \leq \ \delta^2 k^2$, and so 
\begin{equation}
    k_i^2\|q\|_\infty \leq k_i^2 +(1+4\delta)\|k\|_\infty^2 + \|\Delta d/d  \|_\infty.
    \label{eq: q_norm}
\end{equation}

\end{proof}

\bibliographystyle{plain}
\bibliography{references}

\end{document}